\documentclass{article}
\usepackage{amssymb}
\usepackage{amsfonts}
\usepackage{amsmath}
\usepackage{graphicx}
\usepackage{color}
\usepackage[pdftex,bookmarks,colorlinks]{hyperref} 
\hypersetup{colorlinks,%
                    citecolor=blue,%
                      filecolor=blue,%
                      linkcolor=black,%
                      urlcolor=blue,%
                      pdftex}
\usepackage[pdftex]{hyperref}

\def\sl{{\rm{SL}}}
\def\gl{{\rm{GL}}}
\def\gu{{\rm{GU}}}
\def\su{{\rm{SU}}}
\def\psl{{\rm{PSL}}}

\def\psu{{\rm{PSU}}}

\def\o{{\rm{O}}}
\def\pso{{\rm{P \Omega }}}
\def\psizo{{\rm{\Omega }}}
\def\sp{{\rm{Sp}}}
\def\psp{{\rm{PSp}}}
\def\ppsl{ ( {\rm{P}} ) {\rm{SL}}}
\def\ppsp{ ( {\rm{P}} ) {\rm{Sp}}}
\def\ppsu{ ( {\rm{P}} ) {\rm{SU}}}
\def\ppso{({\rm{P}}){\Omega}}

\def\eps{\varepsilon}
\def\s{\sigma}
\def\ak{\bar{K}}
\def\ag{\bar{G}}
\def\at{\bar{T}}

\def\la{\langle}
\def\ra{\rangle}

\newtheorem{theorem}{Theorem}[section]
\newtheorem{lemma}[theorem]{Lemma}

\newtheorem{remark}[theorem]{Remark}
\newtheorem{corollary}[theorem]{Corollary}

\newtheorem{example}[theorem]{Example}
\newtheorem{definition}[theorem]{Definition}

\title{Construction of Curtis-Phan-Tits system in black box classical groups}

\author{Alexandre Borovik\footnote{School of Mathematics, University of Manchester, UK; alexandre.borovik@manchester.ac.uk } and \c{S}\"{u}kr\"{u} Yal\c{c}\i nkaya\footnote{Corresponding author: School of Mathematics and Statistics, University of Western Australia; sukru@maths.uwa.edu.au}}

\begin{document}

\maketitle

\begin{abstract}
We present a polynomial time Monte-Carlo algorithm for finite simple black box classical groups of odd characteristic which constructs all root $\sl_2(q)$-subgroups  associated with the nodes of the extended Dynkin diagram of the corresponding algebraic group. 
\end{abstract}

\begin{footnotesize}
\tableofcontents
\end{footnotesize}

\section{Introduction}
\

Babai and Szemeredi \cite{babai84.229} introduced \textit{black box groups} as an ideal setting for an abstraction of the permutation and matrix group problems in computational group theory. A black box group is a group whose elements are represented by $0-1$ strings of uniform length and the tasks: multiplying group elements, taking inverse and checking whether a string represents a trivial element or not are done by an oracle (or `black box'). A black box group algorithm is an algorithm which does not depend on specific properties of the representation of the given group or how the group operations are performed \cite{seress2003}. 

A black box group $X$ is specified as $X =\langle S \rangle$ for some set $S$ of elements of $X$ and to construct a black box subgroup means to construct some generators for this subgroup. We have $|X|\leqslant 2^N$ where $N$ is the encoding length. Therefore, if $X$ is a classical group of rank $n$ defined over a field of size $q$, then $|X|>q^{n^2}$, and so $O(N)=n^2\log q$.

An important component of black-box group algorithms is the construction of uniformly distributed random elements. In \cite{babai91.164}, Babai proved that 
there is a polynomial time Monte-Carlo algorithm producing ``nearly'' uniformly distributed random elements in black box groups. However, this algorithm is not convenient for practical purposes, especially for matrix groups, as its cost is $O(N^5)$ where $N$ is the input length. A more practical solution is the ``product replacement algorithm'' \cite{celler95.4931}, see also  \cite{pak.01.476, pak01.301}.

In this paper, we complete the black box recognition of a finite group $X$ where $X/O_p(X)$ is a simple classical group of odd characteristic $p$ via the approach introduced in \cite{suko01}. By the availability of a black box oracle for the construction of a centraliser of an involution in  black box groups \cite{altseimer01.1, borovik02.7, bray00.241}, a uniform approach is proposed in \cite{suko01} to recognise a black box group $X$ where $X/O_p(X)$ is a simple group of Lie type of odd characteristic $p$ by utilizing the ideas from the classification of the finite simple groups. The structure of this approach is as follows.
\begin{enumerate}
\item Construct a subgroup $K \leqslant X$ where $K/O_p(K)$ is a long root $\sl_2(q)$-subgroup in $X/O_p(X)$.
\item Determine whether $O_p(X) \neq 1$.
\item Construct all subgroups  $K \leqslant X$ where $K/O_p(K)$ corresponds to root $\sl_2(q)$-subgroups associated with the extended Dynkin diagram of the corresponding algebraic group.  
\end{enumerate}
The task (1) is completed in \cite{suko02}, and (2) is announced independently in \cite{parkerwilson} and \cite{suko02}. The present paper completes the task (3) for classical groups, that is, we prove the following.

\begin{theorem}\label{cpt:main}
Let $X$ be a black box group where $X/O_p(X)$ is isomorphic to a finite simple classical group over a field of odd size $q=p^k>3$. Then there is a Monte-Carlo polynomial time algorithm  which constructs all subgroups corresponding to root $\sl_2(q)$-subgroups of $X/O_p(X)$ associated with the nodes in the extended Dynkin diagram of the corresponding algebraic group.
\end{theorem}

In a subsequent paper \cite{suko05}, we extend Theorem \ref{cpt:main} to all black box groups of Lie type of odd characteristic proving an analogous result for the exceptional groups of Lie type of odd characteristic. We shall note here that this approach can be viewed as a black box analogue of Aschbacher's Classical Involution Theorem \cite{aschbacher77.353} which is the main identification theorem in the classification of the finite simple groups, see \cite{suko01} for a discussion of the analogy between our approach and the Classical Involution Theorem. Besides building an analogy between recognition of black box groups and the classification of the finite simple groups, our approach also answers some interesting questions in computational group theory. For example, it immediately follows from Theorem \ref{cpt:main} that we can construct  representatives of all conjugacy classes of involutions in a classical group $G$ of odd characteristic. Moreover, we can also construct all subsystem subgroups of $G$ which can be read from the extended Dynkin diagram and normalised by some maximally split torus, if $G$ is not a twisted group. In the twisted case, such a torus is of order $(q+1)^n$ where $n$ is the Lie rank of the corresponding algebraic group. A \textit{subsystem subgroup} of a simple group $G$ of Lie type is defined to be a subgroup which is normalised by some maximal torus of $G$. 

There are mainly two types of black box algorithms to recognise a finite group: probabilistic and constructive recognition algorithms. The probabilistic recognition algorithms are designed to determine the isomorphism type of the groups with a user prescribed probability of error, for example, the standard name of a given simple group of Lie type can be computed in Monte-Carlo polynomial time  \cite{altseimer01.1, babai02.383} by assuming an order oracle with which one can compute the order of elements. If successful, the constructive recognition algorithms establish isomorphism between a given black box group and its standard copy. The constructive recognition of black box classical groups is presented in \cite{kantor01.168}. However, they are not  polynomial time algorithms in the input length. They are polynomial in $q$ whereas the input size involves only $\log q$. The size of the field $q$ appears in the running time of the algorithm because unipotent elements (or $p$-elements where $p$ is the characteristic of the underlying field) are needed to construct an isomorphism and the proportion of unipotent elements or, more generally, $p$-singular elements (whose orders are multiple of $p$) is $O(1/q)$ \cite{guralnick01.169}. At the present time, it is still not known how to construct a unipotent element except for random search in a black box group. Later, these algorithms are extended to polynomial time algorithms in a series of papers \cite{ brooksbank03.162, brooksbank01.95, brooksbank06.256, brooksbank08.885} by assuming a constructive recognition of $\sl_2(q)$. We shall note here that our approach is not a constructive recognition algorithm. 

Following our setting in \cite{suko01, suko02}, we assume that the characteristic $p$ of the underlying field is given as an input. However, this assumption can be avoided by using one of the algorithms presented in \cite{kantor02.370}, \cite{ kantor09.802} or  \cite{liebeck07.741}. In our algorithms we do not use an order oracle, instead we assume that a computationally feasible global exponent $E$ is given. Note that one can take $E=|\gl_n(q)|$ for an $n\times n$ matrix group over a field of size $q$.

\section{Root $\sl_2(q)$-subgroups in finite groups of Lie type}
\

Let  $\ag$ denote a connected simple algebraic group over an algebraically closed field of characteristic $p$, $\at$ a maximal torus of $\ag$, $\bar{B}$ a Borel subgroup containing $\at$ and $\bar{\Sigma}$ be the corresponding root system. Let $\bar{N} = N_{\ag}(\at)$ and $W=\bar{N}/\at$ the Weyl group of $\ag$. Let $\s$ be a Frobenius endomorphism of $\ag$ and $\ag_\s$ the fixed point subgroup of $\ag$ under $\s$.  The subgroup $\ag_\s$ is finite, and we denote $G=O^{p^\prime}(\ag_\s)$. 

For each root $r \in \bar{\Sigma}$, there exists a $\bar{T}$-root subgroup of $\bar{G}$. If $\at$ is $\sigma$-invariant, then the map $\s$ permutes these root subgroups and induces an isometry on the Euclidean space $\mathbb{R}\bar{\Sigma}$ spanned by $\bar{\Sigma}$. Let $\Delta$ be a $\langle \s \rangle$-orbit of a root subgroup of $\bar{G}$, then the subgroup $O^{p^{\prime}}(\langle \Delta \rangle_\s)$ is called a $T$-root subgroup of $G$, where $T=\bar{T}_\s$ \cite{seitz83.153}. The root system of the finite group $G$ is obtained by taking the fixed points of the isometry induced from $\s$ on $\mathbb{R}\bar{\Sigma}$ \cite[Section 2.3]{gorenstein1998}, and $G$ is generated by the corresponding root subgroups. A root subgroup is called a long or short root subgroup, if the corresponding root is long or short, respectively. We refer the reader to \cite[Table 2.4]{gorenstein1998} for a complete description of the structure of the root subgroups in finite groups of Lie type.

Let $\Sigma=\{ r_1, \ldots , r_n\}$ and $X_{r_1}, \ldots, X_{r_n}$ be the corresponding root subgroups. Set $M_i = \langle X_{r_i}, X_{-r_i} \rangle$, $Z_i = Z(X_{r_i})$ and $K_i = \langle Z_i, Z_{-i}\rangle$. Then $X_{r_i}$ is a Sylow $p$-subgroup of $M_i$ and $Z_i$ is a Sylow $p$-subgroup of $K_i$. The subgroup $K_i\leqslant G$ is called \textit{long or short root $\sl_2(q)$-subgroup}, if the corresponding root $r_i \in \Sigma$ is a long or short, respectively. Here, $q$ is the order of the centre of a long root subgroup of $G$. Note that if $G\cong \psu_{2n+1}(q)$, then there exists a long root subgroup $X_r$ of order $q^3$ where $M =\langle X_r,X_{-r} \rangle \cong \psu_3(q)$. For this root subgroup we also have $K_r = \langle Z(X_r), Z(X_{-r})\rangle \cong \sl_2(q)$. We have the following fundamental result about long root $\sl_2(q)$-subgroups.

\begin{theorem}{\rm (\cite[Theorem 14.5]{aschbacher77.353})}\label{classical:sl2}
Let $G$ be a finite simple group of Lie type defined over a field of odd order $q>3$ different from $\psl_2(q)$ and $^2G_2(q)$. With the above notation, let $r_i$ be a long root, $K=K_i$, and $\langle z \rangle = Z(K)$. Then
\begin{itemize}
\item[{\rm (1)}] $K\cong \sl_2(q)$.
\item[{\rm (2)}] $O^{p^\prime}(N_G(K)) = KL$, where $[K,L]=1$ and $L$ is the Levi factor of the parabolic subgroup $N_G(Z_i)$.
\item[{\rm (3)}] $K \unlhd C_G(z)''$. Moreover,  if $G$ is not orthogonal, then $N_G(K)=C_G(z)$. 
\end{itemize} 
\end{theorem}

We call the involutions in long root $\sl_2(q)$-subgroups \textit{classical} involutions.

Following the above notation, it is worth to list the short root $\sl_2(q)$-subgroups in classical groups.
\begin{table}[h]\label{short:root:sl2} 
\begin{center}
\begin{tabular}{ccccccccccc} 
$G(q)$               &&&&condition&&&  $K_r =M_r$ \\ \hline
$\psp_{2n}(q)$    &&&&$n\geqslant 3$&&&    $\sl_2(q)$ \\
$\psu_n(q)$        &&&&$n\geqslant 4$&&&   $\psl_2(q^2)$ \\
$\pso_{2n}^-(q)$    &&&&$n\geqslant 2$&&&    $\psl_2(q^2)$ \\
$\psizo_{2n+1}(q)$   &&&&$n\geqslant 2$&&&   $\psl_2(q)$ 
\end{tabular}
\end{center}
\caption{Short root $\sl_2(q)$-subgroups in classical groups \cite[Table 14.4]{aschbacher77.353}.}
\end{table}

\section{Involutions in classical groups}
\

We summarise the conjugacy classes of involutions and their centralisers in simple classical groups in Table \ref{cent:table}. The table is extracted from \cite[Table 4.5.1]{gorenstein1998} for the convenience of the reader. The proofs of the results presented in Table \ref{cent:table} can be found in \cite[Chapter 4]{gorenstein1998}.

\begin{table}[ht]
\caption{Centralisers of involutions in finite simple classical groups} \label{cent:table}
\begin{tabular}{|c|c|c|c|c|} \hline 
$G$  & conditions  &  type & $O^{p^\prime}(C_G(i))$   \\ [.7ex] \hline  
         & &  $t_1$ &$\sl_{n-1}^\eps(q)$       \\ [.7ex]
$\psl_n^\eps(q)$ &  $2 \leqslant k \leqslant n/2$ & $t_k$   &          $ \sl_k^\eps(q) \circ \sl_{n-k}^\eps(q)$               \\[.7ex]

            &  $n$ even                      & $t_{n/2}^\prime$   & $\frac{1}{(n/2,q-\eps)}\sl_{n/2}(q^2)$      \\[.7ex] \hline
                   &                     & $t_1$        & $\psizo_{2n-1}(q)$                             \\[.7ex]
                   &                     & $t_1^\prime$ & $\psizo_{2n-1}(q)$              \\[.7ex]

$\psizo_{2n+1}(q)$ &  $2\leqslant k <n $ & $t_k$            & $\psizo^+_{2k}(q) \times \psizo_{2(n-k)+1}(q)$  \\[.7ex]
$n\geqslant 2$     &  $2\leqslant k <n $ & $t_k^\prime$     & $\psizo^-_{2k}(q) \times \psizo_{2(n-k)+1}(q)$  \\[.7ex]
                  &                     & $t_n$            & $\psizo^+_{2n}(q)$       \\ [.7ex]
                &                     & $t^\prime_n$     & $\psizo^-_{2n}(q)$         \\ [.7ex] \hline 
                   
$\psp_{2n}(q)$     & $1\leqslant k < n/2$ & $t_k$        & $\sp_{2k}(q) \circ_2 \sp_{2(n-k)}(q)$  \\[.7ex]
$n\geqslant2$      &                              & $t_n$        & $\frac{1}{(2,n)}\sl_n(q)$            \\[.7ex]
                &                              & $t_n^\prime$ & $\frac{1}{(2,n)}\su_n(q)$     \\[.7ex]  \hline
             &                       & $t_1$        & $\psizo^\eps_{2n-2}(q)$                \\[.7ex]                           &                       & $t_1^\prime$ & $\psizo^{-\eps}_{2n-2}(q)$              \\[.7ex]    
${\rm P}\psizo^\eps_{2n}(q)$  & $2\leqslant k < n/2$  & $t_k$            & $\psizo^+_{2k}(q) \circ_2 \psizo^\eps_{2(n-k)}(q)$  \\[.7ex]              
$ n\geqslant 4$            & $2\leqslant k < n/2$  & $t_k^\prime$     & $ \psizo^-_{2k}(q) \circ_2 \psizo^{-\eps}_{2(n-k)}(q)$  \\[.7ex]   
                        & $\pso^+_{4m}(q)$  & $t_{n/2}$        & $ \psizo^+_{2m}(q) \circ_2 \psizo^+_{2m}(q)$           \\[.7ex]                     
                     & $\pso^+_{4m}(q)$  & $t_{n/2}^\prime$ & $ \psizo^-_{2m}(q) \circ_2 \psizo^-_{2m}(q) $         \\[.7ex]  
                  & $\pso^+_{4m}(q)$  & $t_{n-1}, \, t_n$        & $\frac{1}{2}\sl_{2m}(q)$                           \\[.7ex]            
               & $\pso^+_{4m}(q)$  & $t_{n-1}^\prime, \, t_n^\prime$ & $\frac{1}{2}\su_{2m}(q)$                           \\[.7ex]
           & $\pso^-_{4m}(q)$  & $t_{n/2}$ & $ \psizo^-_{2m}(q) \times \psizo^+_{2m}(q) $                          \\[.7ex]
        & $\pso^\eps_{2(2m+1)}(q)$ & $t_n$        & $\sl^\eps_{2m+1}(q)$                                      \\[.7ex]  \hline
\end{tabular}
\end{table}    

Let $V$ denote the underlying natural module for classical groups. The involutions $t_k$ in $\psl_n^\eps(q)$ act as  involutions in $\gl_n^\eps(V)$ where the eigenvalue $-1$ has multiplicity $k$. If $n$ is even, then there is an involution of type $t_{n/2}'$ which arises from an element of order 4 in $Z(\gl_{n/2}(q^2))$. Note that $\gl_{n/2}(q^2)$ acts naturally on a totally isotropic subspace of dimension $n/2$ in a unitary geometry. 
 
In $ \psp_{2n}(q)$, the involutions of type $t_k$ for $1\leqslant k < n/2$ represent an element of order 2 in $\sp_{2n}(q)$ whereas an involutions of type $t_n$ and $t_n'$ represent an element $t \in \sp_{2n}(q)$ such that $t^2=-I$ where $I$ is $2n\times 2n $ identity matrix. The eigenvalue $-1$ has multiplicity $2k$ for an involution of type $t_k$, $1\leqslant k< n/2$. If $q \equiv 1 \mbox{ mod } 4$, then an element of $Z(\gl_n(q))$ of order 4 induces an involution in $\psp_{2n}(q)$ which is denoted by $t_n$. Note that $\gl_n(q)$ can be viewed as a stabiliser of a maximal totally isotropic subspace. Similarly, when $q \equiv -1 \mbox{ mod }4$, $\gu_n(q)$ can be embedded in $\sp(V)$ and similar construction induces an involution denoted by $t_n'$ 

In $ \psizo_{2n+1}(q)$, the involutions of type $t_k$, $t_k'$ act as involutions in $\o(V)$ where the eigenvalue $-1$ has multiplicity $2k$. Note that the spinor norm determines whether $-I_{2k}$ belongs to $\psizo(W)$ where $W$ is $2k$ dimensional orthogonal geometry. Indeed, $-I_{2k} \in \psizo_{2k}^+(q)$ if and only if $q^k \equiv 1 \mbox{ mod }4$ and $-I_{2k} \in \psizo_{2k}^-(q)$ if and only if $q^k \equiv -1 \mbox{ mod }4$.

In $\pso_{2n}^\eps(q)$, the involutions of type $t_k$ for $1\leqslant k \leqslant n/2$ act similarly as in $\psizo_{2n+1}(q)$. The involutions of type $t_{n-1}$ and $t_n$ are $O(V)$-conjugate. As in $\psp_{2n}(q)$, $\gl_n(q)$ can be embedded in $O^+(V)$ as a stabiliser of a maximal totally isotropic subspace. If $q\equiv 1 \mbox{ mod } 4$, then the description of an involution of type $t_n$ is same as in $\psp_{2n}(q)$. Similarly $\gu_n(q)$ can be embedded in $\o^\eps_{2n}(q)$ where $\eps = (-1)^n$ and the involutions of type $t_{n-1}'$ and $t_n'$ arises from $Z(\gu_n(q))$ for $\pso_{4m}^+$ when $q \equiv -1 \mbox{ mod } 4$. The description of  the involutions of type $t_n$ in $\pso_{2(2m+1)}^-(q)$ is similar.

\section{Curtis-Phan-Tits presentation}
\

The finite groups of Lie type have a special presentation called the \textit{Steinberg presentation} \cite{steinberg1968} where the generators and relations are given by root subgroups. Steinberg proved that if $G$ is a finite group generated by  the set $\{x_r(t) \mid r \in \Sigma, \, t\in \mathbb{F}_q \}$, where $\Sigma$ is an irreducible root system of rank at least $2$, subject to the relations

\begin{equation}\label{root1}
x_r(t+u)=x_r(t)x_r(u),
\end{equation}
\begin{equation}\label{cheva2}
[x_r(t) , x_s(u) ] = \prod_{ \begin{array}{c}
                                       \gamma= ir + js, \, i,j \in \mathbb{N}^* \\
                                       r, s\in \Sigma, \, r \neq \pm s
                            \end{array}}    x_\gamma(c_{i,j,r,s}t^iu^j),
\end{equation}

\begin{equation}\label{cheva3}
h_r(t) h_r(u) = h_r(tu) \quad tu\neq 0,
\end{equation}
where
$$h_r(t) = n_r(t) n_r(-1),$$
$$n_r(t) = x_r(t) x_{-r}(-t^{-1}) x_r(t),$$
then $G/Z(G)$ is a untwisted simple group of Lie type with root system $\Sigma$, see \cite[Theorem 8, p. 66]{steinberg1968} or \cite[Theorem 12.1.1]{carter1972}. The analogue of the Steinberg presentation holds also for twisted groups of Lie type where the defining relations are more sophisticated, a detailed discussion can be found in \cite[Section 2.4]{gorenstein1998}.

The following theorem (known as the Curtis-Tits presentation) shows that the essential relations in the Steinberg presentation are the ones involving rank 1-subgroups corresponding to fundamental roots in $\Sigma$. Note that, if $G$ is untwisted, then we have 
$$\langle X_r , X_{-r}\rangle \cong \ppsl_2(q)$$ where $X_r =\langle x_r(t) \mid t\in \mathbb{F}_q\rangle$
for any $r \in \Sigma$. Note also that the nodes in the Dynkin diagram are labelled by the fundamental roots. Therefore the Curtis-Tits presentation involves the pairs of fundamental roots which are edges or non-edges in the Dynkin diagram. More precisely;

\begin{theorem}\cite{curtis65.174, tits62.137}\label{curtis-tits} 
Let $\Sigma$ be an irreducible root system of rank at least $3$ with fundamental system $\Pi$ and Dynkin diagram $\Delta$. Let $G$ be a finite group and assume that the followings are satisfied
\begin{enumerate}
\item $G =\langle K_r \mid r \in \Pi \rangle$, $K_r =\langle X_r, X_{-r} \rangle = \ppsl_2(q)$, for all $r \in \Pi$.
\item $H_r = N_{K_r}(X_r) \cap N_{K_r}(X_{-r}) \leqslant N_G(X_s)$ for all $r,s \in \Pi$.
\item $[K_r, K_s]=1$ if $r$ and $s$ are not connected in $\Delta$.
\item $\langle K_r,K_s \rangle \cong \ppsl_3(q)$ if $r$ and $s$ are connected with a single bond.
\item $\langle K_r,K_s \rangle \cong \ppsp_4(q)$ if $r$ and $s$ are connected with a double bond. 
\end{enumerate}
Then there exists a group of Lie type $\widetilde{G}$ with a root system $\Sigma$ and a fundamental system $\Pi$, and a surjective homomorphism $\varphi:G \rightarrow \widetilde{G}$ mapping the $X_{\pm r}$ onto the corresponding fundamental root subgroups of $\widetilde{G}$. Moreover ${\rm ker}\, \varphi \leqslant Z(G) \cap H$ where $H= \langle H_r \mid r \in \Pi \rangle$.
\end{theorem}

\begin{example}{\rm  \cite[p. 72]{steinberg1968}}\label{steinberg-ex}
{\rm Let $G= \sl_n(q)$, $n \geqslant 3$ and $x_{ij}(t) = I + tE_{ij}$ where $E_{ij}$ is the matrix whose $(i,j)$-entry is $1$ and the others are $0$. Then Steinberg-presentation of $G$ is given as follows.
$$G = \langle x_{ij}(t) \mid 1 \leqslant i,j \leqslant n, \,i\neq j, \, t \in \mathbb{F}_q \rangle $$
subject to the following relations
\begin{enumerate}
\item $x_{ij}(t+u) = x_{ij}(t)x_{ij}(u)$, 
\item $[x_{ij}(t), x_{jk}(u)]=x_{ik}(tu)$ if $i,j,k$ are different,
\item $[x_{ij}(t), x_{kl}(u)] = 1$  if $j \neq k , \, i \neq l.$
\end{enumerate}

In the Curtis-Tits presentation of $G$, it is enough to use the generators $x_{ij}(t)$ where $\mid i-j \mid  \leqslant 2$. 
}
\end{example}

In \cite{phan77.67}, Phan proved a similar result for the groups $^2A_n(q)$, $D_{2n}(q)$, $^2D_{2n+1}(q)$, $^2E_6(q)$, $E_7(q)$, $E_8(q)$. His fundamental result is the following.

\begin{theorem}\cite{phan77.67}\label{phan:su}
Let $G$ be a finite group containing subgroups $K_i \cong \su_2(q)$, $q\geqslant 5$, for $i=1,2, \ldots ,n$  and let $H_i$ be a maximal torus of order $q+1$ in $K_i$. Assume that
\begin{itemize}
\item[(P1)] $G = \langle K_i \mid i=1,\ldots , n\rangle$;
\item[(P2)] $[K_i,K_j]=1$ if $|i-j|>1$;
\item[(P3)] $\langle K_i,K_j\rangle \cong \su_3(q)$ and $\langle K_i,H_j\rangle \cong \gu_2(q)$ if $|i-j|=1$; and
\item[(P4)] $\langle H_i, H_j\rangle=H_i\times H_j$ for all $i\neq j$.
\end{itemize}
Then $G$ is isomorphic to a factor group of $\su_{n+1}(q)$.
\end{theorem}

It is clear that the subgroups $K_i$, $i=1,2, \ldots ,n$ in Theorem \ref{phan:su} play the role of the subgroups corresponding to the nodes in the Dynkin diagram of $\psl_{n+1}(q)$ as in its Curtis-Tits presentation. However, they are not root $\sl_2(q)$-subgroups corresponding to the roots in a fixed fundamental root system of $\su_{n+1}(q)$.

Following Tits' geometric approach on the identification of the untwisted groups of Lie type, a new Phan theory is introduced in \cite{bennett03.13}, and Bennet and Shpectorov \cite{bennett04.287} gave a new proof of Phan's theorem, Theorem \ref{phan:su}, with weaker assumptions which also covers the cases $q=2,3,4$. This new approach to Phan's theorem gives  birth to new Phan-type amalgamations for the untwisted groups of Lie type, see \cite{gramlich04.1, gramlich03.358, gramlich06.603} for symplectic groups, \cite{gramlich05.141} for even dimensional orthogonal groups and \cite{bennett07.426} for odd dimensional orthogonal groups.

Let $K_r$, $r\in \Pi$ and $H$ be the subgroups as in Theorem \ref{curtis-tits}. Then we call $(\{K_r \mid r\in \Pi\}; H)$ a \textit{Curtis-Tits system} for $G$ corresponding to the maximal torus $H$. Let $\Pi^*=\Pi \cup \{\alpha\}$ where $\alpha$ is the highest root in $\Pi$. Then  $H\leqslant N_G(K_\alpha)$ where $K_\alpha$ is the corresponding root $\sl_2(q)$-subgroup and we call $(\{K_r \mid r\in \Pi^*\}; H)$ an \textit{extended Curtis-Tits system} for $G$ corresponding to the maximal torus $H$. 

Similarly, we define an \text{extended Phan system} for a group $G$.

\begin{definition}\label{def:phan}
Let $\Sigma$ be an irreducible root system of rank at least $3$ with fundamental system $\Pi$ and Dynkin diagram $\Delta$. Let $\Pi^*=\Pi \cup \{\alpha \}$ where $\alpha$ is the highest root in $\Pi$ and $\Delta^*$ be the extended Dynkin diagram. Let $G$ be a finite group and assume that the followings are satisfied.
\begin{itemize}
\item $G =\langle K_r \mid r \in \Pi \rangle$, $K_r \cong \su_2(q)$.
\item For all $r,s \in \Pi^*$, $H_r \leqslant N_G(K_s)$, $|H_r|=q+1$ and $H=\langle H_r \mid r\in \Pi^*\rangle$ is an abelian group. 
\item $[K_r, K_s]=1$ if $r$ and $s$ are not connected in $\Delta^*$.
\item $\langle K_r,K_s \rangle \cong\ppsu_3(q)$ if $r$ and $s$ are connected with a single bond.
\item $\langle K_r,K_s \rangle \cong \ppsp_4(q)$ if $r$ and $s$ are connected with a double bond.
\end{itemize}
Then $(\{K_r \mid r \in \Pi \};H)$ is called a Phan system and $(\{K_r \mid r \in \Pi^* \};H)$ is called an extended Phan system for $G$.
\end{definition}

In \cite{suko04}, we generalise the Curtis-Tits system to all possible amalgamations in a finite group of Lie type of odd characteristic. In particular, we obtain the following result which elaborates the relation between Phan and Curtis-Tits systems in terms of root $\sl_2(q)$-subgroups and the corresponding maximal torus normalising them.

\begin{theorem}\cite{suko04}\label{gen:curtis-tits}
Let $\ag$ be a simply connected simple algebraic group of type 
$B_n$, $C_n$, $D_{2n}$, $E_7$, $E_8$, $F_4$ or $G_2$ 
over an algebraically closed field of odd characteristic. Let $\s$ be a standard Frobenius homomorphism and $\at$ a $\s$-invariant maximal torus. Let $(\{\ak_r \mid r \in \Pi^* \}; \at)$ be an extended Curtis-Tits System for $\ag$. Then $( \{(\ak_r^g)_\s \mid r \in \Pi^*\} ; (\at^g)_\s)$ is an extended Phan system for $\ag_\s$ where $g \in \ag$ such that $g^{-1}\s(g)\at \in Z(N_{\ag}(\at)/\at)$.
\end{theorem}

Note that the groups listed in Theorem \ref{gen:curtis-tits} are the only simple algebraic groups whose Weyl groups have non-trivial centre. Therefore the finite groups obtained from these groups are the only untwisted groups of Lie type which have Phan system. The same result also holds for the groups $^2A_n(q)$, $^2D_{2n+1}(q)$, $^2E_6(q)$, $^3D_4(q)$ in which case the Frobenius automorphism $\s$ induces a graph automorphism.

\section{Construction of $C_G(i)$ in black box groups}
\

In this section, we recall the construction of the centralisers of involutions in black-box groups following \cite{borovik02.7}, see also \cite{bray00.241}.

Let $X$ be a black-box finite group having an exponent $E=2^km$ with $m$ odd. To produce an involution from a random element in $X$, we need an element $x$ of even order. Then the last non-identity element in the sequence
$$1 \neq x^{m}, \, x^{m2}, \, x^{m2^2}, \, \ldots , x^{m2^{k-1}}, x^{m2^k}=1$$ 
is an involution and denoted by ${\rm i}(x)$. Note that the proportion of elements of even order in classical groups of odd characteristic is at least $1/4$ \cite{isaacs95.139}.

Let $i$ be an involution in $X$. Then, by \cite[Section 6]{borovik02.7}, there is a partial map $\zeta^i = \zeta^i_0 \sqcup \zeta^i_1$ defined by
\begin{eqnarray*}
\zeta^i: X & \longrightarrow &  C_X(i)\\
x & \mapsto & \left\{ \begin{array}{ll}
\zeta^i_1(x) = (ii^x)^{(m+1)/2}\cdot x^{-1} & \hbox{ if } o(ii^x) \hbox{ is odd}\\
\zeta^i_0(x) = {\rm i}(ii^x)  &  \hbox{ if } o(ii^x) \hbox{ is even.}
\end{array}\right.
\end{eqnarray*}

Here $o(x)$ is the order of the element $x\in X$. Notice that, with a given exponent $E$, we can construct $\zeta_0^i(x)$ and $\zeta_1^i(x)$ without knowing the exact order of $ii^x$.

The following theorem is the main tool in the construction of centralisers of involutions in black-box groups.

\begin{theorem} {\rm (\cite{borovik02.7})} \label{dist}
Let $X$ be a finite group and $i \in X$ be an involution. If the elements $x \in X$ are uniformly distributed and
independent in $X$, then 
\begin{enumerate}
\item the elements $\zeta^i_1(x)$ are uniformly distributed and independent in $C_X(i)$ and
\item the elements $\zeta^i_0(x)$ form a normal subset of involutions in $C_X(i)$.
\end{enumerate}
\end{theorem}

By convention, we write $\zeta_0^i(g)=1$ (resp. $\zeta_1^i(g)=1$) when $ii^g$ is of odd order (resp. even order). It is clear from Theorem \ref{dist} that $\langle\zeta_1^i(G)\rangle = C_G(i)$ and $\langle \zeta_0^i(G) \rangle \unlhd C_G(i)$. By \cite[Theorem 5.7]{suko02}$, \langle \zeta_0^i(G) \rangle$ contains the semisimple socle of the centraliser of an involution $i\in G$ for a simple group $G$ of Lie type of odd characteristic except for $G\cong \psp_{2n}(q)$ and the involution of type $t_1$.

Kantor and Lubotzky \cite{kantor90.67} proved that randomly chosen two elements in a finite simple classical group $G$ generate $G$ with probability $\rightarrow 1$ as $|G|\rightarrow \infty$. They also prove an analogous result for the direct product of finite simple classical groups assuming the order of the each direct factor approaches $\infty$. Therefore some reasonable number of random elements generate the centraliser of an involution in finite simple classical groups over large fields with high probability. By Theorem \ref{dist}, we shall use the map $\zeta_1^i$ to produce uniformly distributed random elements in $C_G(i)$. For an arbitrary involution $i \in G$ where $G$ is a finite simple classical group, the proportion of  elements of the form $ii^g$ which have odd order is bounded from below by $c/n$ where $c$ is an absolute constant and $n$ is the dimension of the underlying vector space \cite{parkerwilson}. For the classical involutions in classical groups, such a proportion is proved to be bounded from below by an absolute constant \cite[Theorem 8.1]{suko02}. 

The map $\zeta_0^i$ is also an efficient tool to generate a subgroup containing semisimple socle of the centraliser of an involution. By Lemma 5.5 and Theorem 5.7 in \cite{suko02}, the image of $\zeta_0^i$ generates a subgroup containing semisimple socle of $C_G(i)$ where $G$ is any simple group of Lie type of odd characteristic except that $G\cong \psp_{2n}(q)$ and $i$ is an involution of type $t_1$. For the construction of a centraliser of an involution ${\rm i}(g)$ for some random element $g \in G$ by using only the map $\zeta_0$, we first note that random elements are powered upto \textit{strong} involutions (eigenspace for the eigenvalue $-1$ has dimension between $n/3$ and $2n/3$) with probability at least $c/\log n$ for an absolute constant $c$ \cite{lubeck09.3397}. Moreover, by \cite{prastrong2}, if  $G\cong\gl_n(q)$ and $i \in G$ is a strong involution, then $\zeta_0^i(g)$ is a strong involution with probability at least $c/\log n$ for an absolute constant $c$. By \cite{prastrong}, we have that constant number of strong involutions generate the semisimple socle of the centralisers of involutions with probability $1-1/q^n$. A similar result is expected to hold for the rest of the classical groups.

We shall note here that the map $\zeta_0$ plays a crucial role in our construction of Curtis-Phan-Tits system. Recall that $\zeta_0^i$ produces involutions in $C_G(i)$, and we use $\zeta_0^i$ for a classical involution $i \in G$ to produce a new classical involution $j \in N_G(K)\backslash C_G(K)$ where $K$ is the long root $\sl_2(q)$-subgroup containing $i$. With the long root $\sl_2(q)$-subgroup $L$ containing $j$, we have $\langle K,L\rangle \cong \sl_3^\eps(q)$, see Lemmas \ref{porct:sln}, \ref{porct:orth:I} and \ref{porct:orth:II}. This is the base of our construction. 

The following simple lemma will be used frequently in the sequel.

\begin{lemma}\cite[Lemma 5.1]{suko02}\label{cent:zeta0}
Let $G$ be a finite group and $i \in G$ be an involution. Then the image of $\zeta^i_0$ does not contain involutions from the coset\/ $iZ(G)$.
\end{lemma}

\section{Probabilistic estimates and other results}\

In this section, we obtain estimates that we need for a polynomial time algorithm  constructing Curtis-Phan-Tits systems for black box classical groups. The estimates are far from being sharp, see Lemmas \ref{conn:sln:dist}, \ref{sln:t1}, \ref{conn:orth:dist}, \ref{orth:t1}, and, as some computer experiments in GAP suggests, we believe that the actual probabilities are much bigger.

\begin{lemma}\label{reflexive}
Let $T$ be a torus in $G$ inverted by an involution $i \in G$ and $S=\{ x\in T \mid x\mbox{ is regular and } x=t^2 \mbox{ for some } t\in T\}$.
Then the proportion of elements of the form $ii^g$ for random $g \in G$ is at least 
\[
\frac{|S|^2|C_G(i)|^2}{2|N_G(T)||G|}.
\]
\end{lemma}

\begin{proof}
We follow the same idea in the proof of Lemma 2.9 in \cite{altseimer01.1}, see also Theorem 8.1 in \cite{suko02}. Consider the map 
\begin{eqnarray*}
\varphi: i^G \times i^G & \rightarrow & G \\
(i^g, i^h) & \mapsto & i^gi^h.
\end{eqnarray*}
Let $x\in T$ and $x=t^2$ for some $t \in T$. Since $i$ inverts $T$, 
\[
ii^t = it^{-1}it=tt=x.
\]
Hence the image of $\varphi$ contains all the elements of the form $t^2$ where $t \in T$.

Let $x\in S$, that is, $x$ is regular and $x=t^2$ for some $t \in T$. Then, we claim that $|\varphi^{-1}(x)| \geqslant |S|/2$. Observe that $i^hi^{th} = (ii^t)^h=(tt)^h=x^h=x$  for any $h \in T$. Moreover, since $i$ inverts $T$, $i^{h_1}=i^{h_2}$ for some $h_1,h_2 \in T$ if and only if $h_1=h_2$ or $h_1^2=h_2^2$. Therefore there are at least $|S|/2$ distinct pairs of involutions $(i^h, i^{th})$ which map to $x$. Hence the claim follows.

Let $R$ the set of all regular elements in $G$ whose elements are conjugate to elements in $S$, then
\[
|R| =  |G: N_G(T)||S|,
\]
and the proportion of pairs of involutions which are mapped to $R$ is 
\[
\frac{|\varphi^{-1}(R)|}{|i^G \times i^G|}  \geqslant \frac{|R| |S| |C_G(i)|^2}{2|G|^2} = \frac{|S|^2 |C_G(i)|^2}{2|N_G(T)| |G|}.
\]
The results follow from the identity $i^gi^h = (ii^{hg^{-1}})^g$.\hfill $\Box$
\end{proof}

\begin{lemma}\label{omg}
Let $G$ be a group and $i\in G$ be an involution. Assume that $1\neq j=\zeta_0^i(g)$ for some $g\in G$. Then the proportion of elements of the form $ii^h$ for random $h\in G$ belonging to $C_G(j)$ is at most 
$1/|C_G(i)|$.
\end{lemma}

\begin{proof}
By the definition of the map $\zeta_0^i(g)$, $i \in C_G(j)$ which implies that $ii^hj=jii^h$ if and only if $i^hj=ji^h$. Since the number of conjugates of $i$ is $|G|/|C_G(i)|$, the result follows. \hfill $\Box$
\end{proof}

\subsection{Groups of type $A_{n-1}$}\label{sec:an}
\

\begin{lemma}\label{conn:sln}
Assume that $G \cong \psl_n^\eps(q)$ where $n\geqslant 3$ and $n\neq 4$. Let $K$ be a long root $\sl_2(q)$-subgroup of $G$ and $i$ be the unique classical involution in $K$. Assume also that $\zeta_0^i(g) \neq 1$ for some $g \in G$. Then $\zeta_0^i(g) \notin C_G(K)$ if and only if  $\zeta_0^i(g)$ is a classical involution in $G$. Moreover, $\zeta_0^i(g) \in N_G(K)$.
\end{lemma}

\begin{proof}
We prove the claim when $G\cong \psl_n(q)$ and the case $G \cong \psu_n(q)$ is analogous. 

Assume that $\zeta_0^i(g) \neq 1$ for some $g \in G$. If $n\geqslant 5$, then the subgroup $\langle K, K^g \rangle$ is contained in a subgroup $L \cong \sl_4(q)$ and the involutions in $L$ are either classical in $G$ or the central involution in $L$. Hence if $\zeta_0^i(g) \notin C_G(K)$, then $\zeta_0^i(g)$ is a classical involution. Conversely, assume that $\zeta_0^i(g)$ is a classical involution in $G$. Notice that the only classical involutions in $L$ which commute with $K$ belong to $iZ(L)$. By Lemma \ref{cent:zeta0}, $\zeta_0^i(g) \notin iZ(L)$ for any $g \in L$. Hence  $\zeta_0^i(g) \notin C_G(K)$.

If $G\cong\psl_3(q)$, then all involutions are conjugate and classical. Thus $\zeta_0^i(g)$ is a classical involution. Conversely, the only classical involution in $G$ which commutes with $K$ is the involution $i\in K$, and $\zeta_0^i(g) \neq i$ for any $g \in G$ by Lemma \ref{cent:zeta0}.

By Theorem \ref{classical:sl2}, $C_G(i) = N_G(K)$ so $\zeta_0^i(g) \in N_G(K)$. \hfill $\Box$
\end{proof}

\begin{remark}\label{nonclass-sl4}
{\rm Assume that $G\cong \psl_4(q)$. If $i$ is a  classical involution in $G$, then the involutions of the form $\zeta_0^i(g)$ are not necessarily classical involutions. However, it is clear that the image of $\zeta_0^i(G)$ contains classical involutions. There are three conjugacy classes of involutions which are of type $t_1, \, t_2$ (classical)  and $t_2^\prime$ in $G$. Note that involutions of type $t_2^\prime$ exists in $G$ exactly when $q \equiv -1  \mbox{ mod } 4$ and they are conjugate to
$$j= \left[
         \begin{array}{cc}
            0  & I_2       \\
             -I_2 & 0   \\
          \end{array}
      \right] Z. $$
Assume that
$$i= \left[
         \begin{array}{cccc}
           -I_2 &  0    \\
             0 & I_2  \\
          \end{array}
      \right]Z $$
then $i$ is conjugate to 
$$t= \left[
         \begin{array}{cccc}
             0 &  -I_2    \\
             -I_2 &   0  \\
          \end{array}
      \right]Z, $$
say $t=i^g$ for some $g \in G$. Now $j=it=ii^g=\zeta_0^i(g)$ is an involution in $\psl_4(q)$ which is of type $t_2^\prime$ in $\psl_4(q)$. 
}
\end{remark}

\begin{lemma}\label{porct:sln}
Assume that $G\cong \psl_n^\eps(q)$ where $n\geqslant 4$. Let $K_1$ and $K_2$ be two long root $\sl_2(q)$-subgroups of $G$ containing the classical involutions $i_1$ and $i_2$, respectively. If $i_1$ and  $i_2$ commute with each other and $i_2 \notin C_G(K_1)$, then $\langle K_1, K_2\rangle \cong\sl_3^\eps(q)$. If $G\cong \psl_3^\eps(q)$, then $\langle K_1, K_2\rangle = G$.
\end{lemma}

\begin{proof}
Let $G \cong \sl_n(q)$, $n\geqslant 4$, and $V$ be the natural module for $G$. Let $V=V_-^1 \oplus V_+^1 =V_-^2 \oplus V_+^2 $ where $V_\pm^1$ and $V_\pm^2$ are the eigenspaces of the involutions $i_1$ and $i_2$ corresponding to the eigenvalues $\pm 1$, respectively. We assume that dim$ V_-^1 =$ dim$V_-^2=2$ since $i_1$ and $i_2$ are classical involutions. Notice that $\langle i_1, i_2 \rangle < \sl(V_-^1 + V_-^2)$. Since $i_2 \in C_G(i_1)$, we have $i_2 \in N_G(K_1)$ by Theorem \ref{classical:sl2} so $i_2$ leaves invariant the subspaces $V_-^1$, $V_+^1$. Moreover, $[i_2,V_-^1]\neq 0$ since $i_2 \notin C_G(K_1)$. Now, if dim$[i_2, V_-^1]=2$, then $i_1 = i_2$. Therefore we have dim$[i_2, V_-^1]=1$ which implies that dim$(V_-^1+V_-^2)=3$ and $\langle K_1, K_2 \rangle \cong\sl_3(q)$. The proof is analogous for $G\cong \psu_n(q)$ and $\psl_3^\eps(q)$. \hfill $\Box$
\end{proof}

\begin{lemma}\label{conn:sln:dist}
Assume that $G \cong \psl_n^\eps(q)$ where $n\geqslant 3$. Let $K$ be a long root $\sl_2(q)$-subgroup of $G$ and $i$ be the unique involution in $K$. Then the probability of finding an element $g\in G$, where $\zeta_0^i(g)$ is a classical involution, is at least $1/750(1-2/q)$.
\end{lemma}

\begin{proof}
Assume first that $n \geqslant 5$. For $g \in G$, the subgroup $\langle i, i^g \rangle$ is contained in a subgroup $L$ isomorphic to $\sl_4^\eps(q)$. Indeed, for a random element $g \in G$, we have $L=\langle K,K^g \rangle \cong \sl_4^\eps(q)$  with probability at least $1-2/q$, see Theorem \ref{thm:type}. Therefore, it is enough to find the estimate in $\sl_4(q)$ and $\su_4(q)$. 

Assume that $L\cong \sl_4(q)$. Then $L$ has a subgroup of the form $N=N_1\times N_2$ where $N_1\cong N_2\cong \sl_2(q)$ and $i$ acts as an involution of type $t_1$ on both $N_1$ and $N_2$. It is clear that $i$ inverts a torus of order $q\pm 1$ on $N_1$ and $N_2$. 

Assume that $q \equiv 1 \mbox{ mod } 4$ and consider a torus $T=T_1\times T_2 \leqslant N = N_1\times N_2$ where $T$ is inverted by $i$ and $|T_1|=q-1$ and $|T_2|=(q+1)/2$. Observe that $T$ is uniquely contained in a maximal torus of order $(q-1)^2(q+1)$. Since $(q+1)/2$ is odd, the involution in $T$ belongs to $N_1$ and hence it is a classical involution. It is clear that this involution does not centralise $K$. Now, observe that $|N_L(T)|= 4 (q-1)^2(q+1)$, $|C_L(i)|=q^2(q+1)^2(q-1)^3$ and $|L|=q^6(q^2-1)(q^3-1)(q^4-1)$.  Setting $$S=\{x\in T \mid x \mbox{ is regular, } |x| \mbox{ is even and } x=t^2 \mbox{ for some }t\in T \}$$ we have $|S|\geqslant |T|/4 = (q^2-1)/8$. By Lemma \ref{reflexive}, $ii^g$ has even order and $\zeta_0^i(g)$ is a classical involution with probability at least
\begin{eqnarray*}
\frac{|S|^2|C_L(i)|^2}{2|N_L(T)||L|}& =& \frac{\frac{(q^2-1)^2}{64} q^4(q+1)^4(q-1)^6 }{8 (q-1)^2(q+1)q^6(q^2-1)(q^3-1)(q^4-1)}\\
&=&\frac{1}{512} \frac{(q^2-1)^2}{q^4+q^2} \frac{q^2-1}{q^2+q+1}\\
& \geqslant &\frac{1}{750}
\end{eqnarray*}
since $q\geqslant 5$.

If $q \equiv -1 \mbox{ mod } 4$, then we consider a torus $T=T_1\times T_2 \leqslant N$ where $T$ is inverted by $i$ and $|T_1|=(q-1)/2$ and $|T_2|=q+1$. The rest of the proof is same as above.

The proof is the same for the groups $L\cong \su_4(q)$.

The computations in the case $L\cong \psl_4(q)$ are analogous, namely, we consider the central product $N=N_1\circ_2 N_2$ and apply the above arguments. If $L\cong \psl_3(q)$, then the only involution in $C_L(i)$ which centralise the component $\sl_2(q)$ is the involution $i$ itself. Therefore, for any $g \in L$, if $\zeta_0^i(g)\neq 1$ or equivalently $ii^g$ has even order, then $\zeta_0^i(g)$ does not centralise $K$ since $\zeta_0^i(g)\neq i$  by Lemma \ref{cent:zeta0}. The proportion of the elements $g \in L$ such that $ii^g$ has even order is at least $1/750$ by the similar computations. 

The cases $\psu_n(q)$ for $n=3,4$ are similar. \hfill $\Box$
\end{proof}

\begin{lemma}\label{sln:t1}
Let $G\cong \psl_n^\eps(q)$, $n\geqslant 3$, and $i$ be an involution of type $t_1$. Then $ii^g$ has even order with probability at least 1/30 for a random element $g\in G$. Moreover, $\zeta_0^i(g)$ is a classical involution in $G$.
\end{lemma}

\begin{proof}
Observe that $\langle i, i^g \rangle \leqslant L$ where $L\cong \sl_2(q)$.  Therefore it is enough to find the estimate in $L$. Observe also that $i$ inverts a torus $T\leqslant L$ of order $q\pm 1$. Assume that $q \equiv 1 \mbox{ mod } 4$, the other case is analogous. Then take a torus $T$ of order $q-1$ which is inverted by $i$. Note that $|N_G(T)|=2|T|$ and $|C_G(i)|=2(q-1)$. Let 
$$S=\{x\in T \mid x \mbox{ is regular, } |x| \mbox{ even, and } x=t^2 \mbox{ for some } t\in T \}.$$
Since $T$ is cyclic and all elements are regular, $|S|\geqslant |T|/4$. By Lemma \ref{reflexive}, $ii^g$ has even order with probability at least
\begin{eqnarray*}
\frac{|S|^2|C_L(i)|^2}{2|N_L(T)||L|} &\geqslant& \frac{4(q-1)^4} {64q(q-1)^2(q+1)}\\
&=&\frac{(q-1)^2}{16q(q+1)}\\
&\geqslant &\frac{1}{30}
\end{eqnarray*}
since $q\geqslant 5$. Since $\zeta_0^i(g)$ is an involution and $\zeta_0^i(g) \in L$, it must be a classical involution. \hfill $\Box$
\end{proof}

\subsection{Groups of type $B_n$ and $D_n$}\label{sec:bndn}
\

In this section we deal with all types of orthogonal groups simultaneously and we simply write $\pso_n^\eps(q)$, $\eps =\pm$, to denote orthogonal groups of any type. If $n$ is even, $\pso_n^+(q)$ (resp. $\pso_n^-(q)$) is the orthogonal group where the underlying vector space has Witt index $n/2$ (resp. $n/2-1$). If $n$ is odd, $\eps$ should be ignored. 

\begin{lemma}\label{conn:orth}
Assume that $G\cong \pso_n^\eps(q)$ where $n\geqslant 7$. Let $K$ be a long root $\sl_2(q)$-subgroup of $G$ and $i$ be the unique classical involution in $K$. If $1 \neq \zeta_0^i(g) \notin C_G(K)$ for some $g \in G$,  then $\zeta_0^i(g)$ is an involution of type $t_1, t_2 $ (classical), $t_3$ or $t_4$ (in $\pso_8^+(q)$). 
\end{lemma}

\begin{proof}
Let $V$ be the natural module for $G\cong \psizo_n^\eps(q)$ and $V_\pm$ be the eigenspaces of the involution $i$ for the eigenvalues $\pm 1$. Then dim$V_-=4$.    Observe that $\langle i , i^g \rangle <L\cong \psizo(V_- +V_-^g)$ and dim$(V_-+V_-^g) \leqslant 8$. Therefore the involution $\zeta_0^i(g) = {\rm i}(ii^g)$ is of type $t_1,t_2,t_3$ or $t_4$. If $n\geqslant 9$ and dim$(V_-+V_-^g)=8$, then $i$ commutes with $i^g$ and $\zeta_0^i(g) = ii^g \in C_L(K)$. Note that if $G\cong \psizo_7(q)$, then this case does not occur. If $G\cong \pso_8^+(q)$, then the involutions of type $t_3$ and $t_4$ have orders 4 in $\psizo_8^+(q)$.  
\hfill $\Box$
\end{proof}

\begin{remark}\label{fassorth}
{\rm Let $G\cong \pso_8^+(q)$ and $K=K_1$ be a long root $\sl_2(q)$-subgroup containing the classical involution $i$. Then 
$$C_G(i)=(((K_1\circ_2K_2)\circ_2(K_3\circ_2K_4)) \rtimes \langle j_1,j_2\rangle)\rtimes\langle t\rangle$$
where $K_s \cong \sl_2(q)$ for $s=1,\ldots,4$. Here, $j_1$ (resp. $j_2$) are involutions of type $t_1$ interchanging $K_1$ and $K_2$ (resp. $K_3$ and $K_4$), and $t$ is a classical involution interchanging $K_1\circ_2K_2$ and $K_3\circ_2K_4$. Notice that $j=j_1j_2$ is a classical involution. Therefore, unlike in the case of $\ppsl_n^\eps(q)$, not all classical involutions in $C_G(i)$ belong to $N_G(K)$, see Theorem \ref{classical:sl2}. Moreover, since $j$ and $t$ are classical involutions, there exist $g_1, g_2 \in G$ such that $j=i^{g_1}$ and $t=i^{g_2}$, and $\zeta_0^i(g_1)=ij, \zeta_0^i(g_2)=it \notin N_G(K)$. However if a classical involution  $z\in C_G(i)$ does not belong to $N_G(K)$, then $N=\langle K,K^z \rangle\cong \sl_2(q) \circ_2 \sl_2(q)$.  To decide whether a classical involution in $C_G(i)$ belongs to $N_G(K)$, we check whether the subgroup  $N$ contains elements of order dividing $q^2-1$ but not $q-1$ and $q+1$. 
 }
\end{remark}

\begin{lemma}\label{porct:orth:I}
Assume that $G\cong \pso_n^\eps(q)$ where $n\geqslant 7$. Let $K_1$ and $K_2$ be two long root $\sl_2(q)$-subgroups of $G$ containing the classical involutions $i_1$ and $i_2$, respectively. If $i_1$ and $i_2$ commute with each other and $i_2 \in N_G(K_1)\backslash C_G(K_1)$, then $\langle K_1, K_2\rangle \cong  \sl_3(q)$ or $\su_3(q)$. 
\end{lemma}

\begin{proof}
Let $G \cong \psizo_n^\eps(q)$ and $V$ be the natural module for $G$. Let $V=V_-^1 \oplus V_+^1 =V_-^2 \oplus V_+^2 $ where $V_\pm^1$ and $V_\pm^2$ are the eigenspaces of the involutions $i_1$ and $i_2$ corresponding to the eigenvalues $\pm 1$, respectively. We assume that dim$ V_-^1 =$ dim$V_-^2=4$ since $i_1$ and $i_2$ are classical involutions. 

Since $i_2 \in N_G(K_1)$, $i_2$ induces an involution on $\psizo(V_-^1)$ and the induced quadratic form on $W=V_-^1 +V_-^2$ is  non-degenerate. Moreover, since $i_1$ and $i_2$ are commuting with each other, we have $\mbox{dim}(V_-^1\cap V_-^2)=0,2$ or $4$ which implies that $\mbox{dim}(W)=4,6$ or $8$. It is clear that $\langle i_1, i_2 \rangle < \langle K_1, K_2 \rangle \leqslant \psizo(W)$.  If $\mbox{dim}W = 4$, then $V_-^1=V_-^2$ and $i_1=i_2$. Moreover, if $\mbox{dim}W=8$, then $V_-^1 \cap V_-^2 = \{0\}$ and $i_2 \in C_G(K_1)$. Note that this case does not happen when $n=7$. Hence $\mbox{dim}W=6$. Now since $\pso(W) = \pso_6^\pm(q) \cong \psl_4^\eps(q)$, the result follows from Lemma \ref{porct:sln}. \hfill $\Box$
\end{proof}

\begin{lemma}\label{porct:orth:II}
Assume that $G\cong \pso_n^\eps(q)$ where $n\geqslant 9$ or $G\cong \pso_8^+(q)$. Let $K_1, K_2, K_3$ be long root $\sl_2(q)$-subgroups of $G$ containing the classical involutions $i_1, i_2, i_3$, respectively. Assume also that $[K_1,K_3]=1$, $i_2 \in (N_G(K_1)\cap N_G(K_3)) \backslash$ $(C_G(K_1) \cup C_G(K_3))$ and the involutions $i_k$, $k=1,2,3$, mutually commute with each other. Then,
\begin{itemize}
\item[(1)] if $\langle K_1,K_2 \rangle \cong \sl_3(q)$, then $\langle K_2, K_3 \rangle  \cong \sl_3(q)$, or
\item[(2)] if $\langle K_1,K_2 \rangle \cong \su_3(q)$, then $\langle K_2, K_3 \rangle \cong \su_3(q)$.
\end{itemize}
\end{lemma}

\begin{proof}
Let $G\cong \psizo_n^\eps(q)$, $n\geqslant 9$, and $V$ be the natural module for $G$. Let $V_\pm^k$ be the eigenspaces of the involutions $i_k$, $k=1,2,3$, corresponding to the eigenvalues $\pm 1$. Since $i_2 \in (N_G(K_1)\cap N_G(K_3)) \backslash (C_G(K_1) \cup C_G(K_3))$, $W=(V_-^1+V_-^2+V_-^3)=(V_-^1+V_-^3)$ and $W$ is an orthogonal $8$-space with Witt index 4. Moreover $\langle K_1,K_2,K_3 \rangle \leqslant\psizo(W)$.

 Observe that $\langle K_1,K_2 \rangle \leqslant \psizo(W_1)$ and $\langle K_2,K_3 \rangle \leqslant \psizo(W_2)$ where $W_1=(V_-^1 +V_-^2)$ and $W_2=(V_-^2 +V_-^3)$. By the proof of Lemma \ref{porct:orth:I}, $W_1$ and $W_2$ are orthogonal $6$-spaces with Witt indices 2 or 3. Hence $W_1 = V_-^2 \perp U$ where $U<V_-^1$ is either a hyperbolic plane or it does not contain any singular vectors. Moreover, $W=U\perp W_2$ since $W_1\cap W_2 = V_-^2$.
 
 If $\langle K_1,K_2 \rangle \cong \sl_3(q)$, then it is clear that $W_1$ is an orthogonal $6$-space with Witt index 3, and so $U$ is a hyperbolic plane. Since $W$ has Witt index 4 and $W=U\perp W_2$, $W_2$ is also an orthogonal $6$-space with Witt index 3. Thus since $\langle K_2,K_3 \rangle \leqslant \psizo(W_2)$ and $\pso(W_2)\cong \pso^+_6(q) \cong \psl_4(q)$, we have $\langle K_2,K_3\rangle \cong \sl_3(q)$ by Lemma \ref{porct:sln}. 
 
 If $\langle K_1,K_2 \rangle \cong \su_3(q)$, then $W_1$ is an orthogonal $6$-space with Witt index 2, and so $U$ does not contain any singular vectors. Since $W$ has Witt index 4 and $W=U\perp W_2$, $W_2$ is also an orthogonal $6$-space with Witt index 2. Thus since $\langle K_2,K_3 \rangle \leqslant \psizo(W_2)$ and $\pso(W_2)\cong \pso^-_6(q) \cong \psu_4(q)$, we have $\langle K_2,K_3\rangle \cong \sl_3(q)$ by Lemma \ref{porct:sln}. 

The proof for $\pso_8^+(q)$ is similar. \hfill $\Box$ 
\end{proof}

\begin{lemma}\label{conn:orth:dist}
Assume that $G\cong \pso_n^\eps(q)$ where $n\geqslant 7$. Let $K$ be a long root $\sl_2(q)$-subgroup of $G$ and $i$ be the unique classical involution in $K$. Then the probability of finding an element $g \in G$ where $\zeta_0^i(g)$ is a classical involution and $\zeta_0^i(g) \in N_G(K)$ 
is bounded from below by $(1/2^{16}-1/q^{11})(1-2/q)$.
\end{lemma}

\begin{proof} Assume first that $n\geqslant 9$. Take $g \in G$ and consider the subgroup $L= \langle K, K^g \rangle$. By Theorem \ref{thm:type}, $L\cong \psizo_8^+(q)$ with probability at least $1-2/q$. Since $\zeta_0^i(g) \in L$, it is enough to find the estimate in $\psizo_8^+(q)$. Assume now that $L\cong \psizo_8^+(q)$, then $L$ contains a subgroup of the form $N=N_1\times N_2$ where $ N_1\cong N_2\cong\psizo_4^+(q)$ and $i \in N_L(N_1) \cap N_L(N_2)$ acting as an involution of type $t_1$ on both $N_1$ and $N_2$. Observe that an involution of type $t_1$ in $\psizo_4^+(q)$ inverts a torus of order $(q\pm 1)^2/2$. 

Assume that $q \equiv 1 {\mbox{ mod }} 4$ and consider a torus $T=T_1 \times T_2 \leqslant N=N_1\times N_2$ where $T$ is inverted by $i$ and $|T_1|=(q-1)^2/2$ and  $|T_2|=(q+1)^2/4$. Observe that $T$ is a maximal torus of $L$. Since $(q+1)^2/4$ is odd, involutions in $T$ belong to $T_1 <N_1$ and hence they are of type $t_1$ or $t_2$ in $L$.  Observe also that the torus $T_1= \frac{1}{2}P_1P_2$ where $|P_1|=|P_2|=q-1$. Hence an element $g=(g_1, g_2) \in T_1$ powers upto an involution of type $t_2$ in $L$ if and only if the 2-heights of $g_1$ and $g_2$ are same. Now it is easy to see that the probability of a random element which powers upto an involution of type $t_2$ is at least $1/4$.  Now we have $|N_L(T)|=32|T|$. Moreover, $|C_L(i)|=4|\psizo_4^+(q)|^2=q^4(q^2-1)^4$ and $|L|=q^{12}(q^6-1)(q^4-1)^2(q^2-1)/2$. Let $S$ be the regular semisimple elements of even order of the form $t^2$ for some  $t\in T$. Then $|S| \geqslant |T|/4$. Hence, by Lemma \ref{reflexive}, $ii^g$ has even order and $\zeta_0^i(g)$ is a classical involution for a random element $g \in L$ with probability at least 
\begin{eqnarray*}
\frac{1}{4}\cdot \frac{|S|^2|C_L(i)|^2}{2|N_L(T)||L|}& \geqslant &\frac{1}{8} \frac{\frac{(q^2-1)^4}{1024}q^8(q^2-1)^8 }{ \frac{32(q^2-1)^2}{8}q^{12}(q^6-1)(q^4-1)^2(q^2-1)}\\
&=&\frac{1}{32768}\cdot  \frac{(q^2-1)^2}{q^4}\cdot \frac{(q^2-1)^2}{(q^2+1)^2}\cdot \frac{(q^2-1)^2}{q^4+q^2+1} \\
&\geqslant&\frac{1}{32768}\cdot \frac{8}{9}\cdot \frac{6}{8}\cdot \frac{6}{7}\\
& \geqslant & \frac{1}{2^{16}}.
\end{eqnarray*}

Now we shall find an upper bound for the proportions of elements $g \in L$ where $\zeta_0^i(g)$ is a classical involution and $\zeta_0^i(g) \notin N_L(K)$. Setting $K=K_1$, we have $C_L(i)=(K_1\circ_2K_2) \times (K_3\circ_2K_4)\rtimes \langle j_1, j_2\rangle$. Recall that $j_1$ and $j_2$ are involutions of type $t_1$ in $L$ commuting with each other (see Remark \ref{fassorth}). The involution $j_1$ (resp. $j_2$)  interchanges $K_1$ and $K_2$ (resp. $K_3$ and $K_4$) and fixes $K_3$ and $K_4$ (resp. $K_1$ and $K_2$). Hence $j=j_1j_2$ is a classical involution since it is a product of two commuting involutions of type $t_1$ acting on disjoint subspaces. Clearly $j \notin N_L(K)$. Moreover, the only classical involutions in $C_L(i)$ which do not belong to $N_L(K)$ are $j$ and $jz$ where $z \in Z(C_L(i))$ is an involution. By Lemma \ref{omg}, the proportion of elements $g\in L$ satisfying $\zeta_0^i(g)=j$ or $jz$ is at most $4/|C_L(i)|<1/q^{11}$. Thus $\zeta_0^i(g)$ is a classical involution belonging to $N_G(K)$ with probability at least $(1/2^{16}-1/q^{11})(1-2/q)$.
 
If $G\cong \pso_8^+(q)$, then $L=G$ with probability at least $1-2/q$ by Theorem \ref{thm:type}. By the same computations as above, we have $\zeta_0^i(g)$ is classical with probability at least $1/2^{16}$. Note that, by Remark \ref{fassorth}, there is another classical involution in $ C_L(i)$ which interchanges $K_1 \circ_2 K_2$ and $K_3\circ_2 K_4$. However the same computations above yield the same estimate.

Assume that $G\cong \pso_8^-(q)$. Then $L=G$ with probability at least $1-2/q$ by Theorem \ref{thm:type}. In this case $L$ contains a subgroup of the form $N= N_1 \times N_2$ where $N_1 \cong \psizo_4^-(q)$ and $N_2\cong \psizo_4^+(q)$. Let $i=(j_1,j_2) \in N$ be an involution where $j_1 \in N_1$ and $j_2\in N_2$ are involutions of type $t_1$ in $L$. Now $j_1$ inverts a torus of order $(q^2\pm1)/2$ in $N_1$ and $j_2$ inverts a torus of order $(q\pm1)/2$ in $N_2$. 
Hence, by taking a torus of order $(q^2+1)/2$ in $N_1$ and a torus of order $(q-1)/2$ or $(q+1)/2$ in $L_2$ depending on $q \equiv 1$ or 3 $\mbox{mod } 4$, respectively, the proof follows from the same computations as above.

If $G\cong \psizo_7(q)$,  then $L=G$ with probability at least $1-2/q$ by Theorem \ref{thm:type}. Consider the subgroup $N=N_1 \times N_2 \leqslant L$ where $N_1\cong \psizo_4^-(q)$ and $N_2\cong \psizo_3(q)$. Let $j_1$ and $j_2$ be involutions of type $t_1$ in $N_1$ and $N_2$, respectively. Then $j_1$ inverts a torus of order $(q^2\pm1)$ and $j_2$ inverts a torus of order $(q\pm1)$. The result follows from the same computations. \hfill $\Box$
\end{proof}

\begin{lemma}\label{twin}
Let $G\cong \pso_n^\eps(q)$, $n\geqslant 5$ and $K$ be a long root $\sl_2(q)$-subgroup in $G$. Let $i$ be the unique involution in $K$. Then the proportion of elements $g \in C_G(i)$ such that $\langle K,K^g\rangle \cong \sl_2(q)\circ_2\sl_2(q)$ is at least 1/8. 
\end{lemma}
\begin{proof}
Recall that $C_G(i)''=K\tilde{K}L$ where $L\cong \psizo_{n-4}^\eps(q)$. By \cite[Table 4.5.1]{gorenstein1998}, there exists an involution $t \in C_G(i)$ which interchanges $K$ and $\tilde{K}$. Hence the elements which belong to the coset $tC_G(i)''$ interchanges $K$ and $\tilde{K}$ and the result follows from the fact that $|C_G(i):C_G(i)''|\leqslant 8$.\hfill $\Box$
\end{proof}

\begin{lemma}\label{orth:t1}
Let $G\cong \pso_n^\eps(q)$, $n\geqslant 5$, and $i$ be an involution of type $t_1$. Then the probability of finding an element $g\in G$ such that $\zeta_0^i(g)$ is a classical involution is at least 1/960. 
\end{lemma}
\begin{proof}
By the proof of \cite[Lemma 4.12 (i)]{kantor01.168}, $\langle i, i^g \rangle \leqslant L\cong \psizo_4^+(q)$ with probability at least 1/32. Since $i$ acts as an involution of type $t_1$ on the components of $L$, the result follows from Lemma \ref{sln:t1}. \hfill $\Box$.
\end{proof}

\subsection{Groups of type $C_n$}\label{sec:cn}\

Recall that the group $G=\psp_{2n}(q)$ contains  maximal tori of order $(q^n-1)/2$ and $(q^n+1)/2$ corresponding to maximal positive and negative cycle of length $n$ in the Weyl group, respectively. We call these tori \textit{maximal twisted tori} and write $\frac{1}{2}T_{q^n\pm 1}$.

\begin{lemma}\cite[Lemma 2.13]{altseimer01.1}\label{sptn:tori}
The involutions in maximal twisted tori $\frac{1}{2}T_{q^n\pm 1}$ are of type $t_n$.
\end{lemma}

\begin{lemma}\label{sptn:dist}
The number of regular elements belonging to a maximal twisted torus is at least $\frac{1}{5n}|G|$.
\end{lemma}
\begin{proof}
This is Lemma 2.3 in the corrected version of \cite{altseimer01.1}.\hfill $\Box$
\end{proof}

\begin{lemma}\label{spn:short}
Assume that $G\cong \psp_{2n}(q)$ and $K$ be a short root $\sl_2(q)$-subgroup of $G$. 
\begin{enumerate}
\item 
If $n\geqslant 3$, then $K\cong \sl_2(q)$ and $C_G(i)^{\prime \prime} \cong \sp_4(q) \circ_2 \sp_{2n-4}(q)$ where $i$ is the unique involution in $K$.
\item If $n=2$, then $K\cong \psl_2(q)$.
\end{enumerate}
\end{lemma}

\begin{proof}
(1) Since $n\geqslant 3$, $K\cong \sl_2(q)$ by Table \ref{short:root:sl2}. Let $V$ be the underlying symplectic geometry. Note that $K$ acts irreducibly on a totally isotropic subspace of dimension 2. Let $i$ be the involution in $K$ then $K \leqslant \sp(V_-)$ where $V_-$ be the eigenspace of $i$ corresponding to the eigenvalue $-1$. Since $V_-$ is non-degenerate, $\sp(V_-) =\sp_4(q)$ and $C_G(i)'' \cong \sp_4(q)\circ_2\sp_{2n-4}(q)$. 

(2) Since $\psp_4(q)\cong \psizo_5(q)$, the result follows from Table \ref{short:root:sl2}. 
\hfill $\Box$
\end{proof}

\begin{lemma}\label{conn:sp4}
Let $G \cong \psp_4(q)$ and $i\in G$ be an involution of type $t_2$, then the probability of producing a classical involution $j \in C_G(i)$ by the map $\zeta_0^i$ which does not centralise $K$ is bounded from below by the constant $1/768$.
\end{lemma}

\begin{proof}
Let $G=\sp_4(q)$ and $V$ be the natural module for $G$. Write $V=V_1 \perp V_2$ where $V_1$ and $V_2$ are hyperbolic planes. Then $\sp(V_1) \cong \sp(V_2)\cong \sl_2(q)$. Consider the tori $T_1 \leqslant \sp(V_1)$ and $T_2 \leqslant \sp(V_2)$ where $|T_1|=|T_2|=q- 1$ or $q+1$  when $q \equiv1 \mbox{ mod 4}$ or $q \equiv -1 \mbox{ mod 4}$, respectively. It is clear that the involutions in $T_1$ and $T_2$ are classical involutions.

Now let $G\cong \psp_4(q)$ and consider the image $\bar{T}$ of $T=T_1\times T_2$ in $G$. We have $\bar{T} =\frac{1}{2}(T_1\times T_2)$. By \cite[Lemma 2.8]{altseimer01.1}, there exists $j \in i^G$ such that $\bar{T}$ is inverted by $j$. Let 
\[
S=\{x\in \bar{T} \mid x \mbox{ is regular, } |x| \mbox{ even, and } x=t^2 \mbox{ for some } t\in \bar{T} \}.
\]
Then $|S| \geqslant |\bar{T}|/8$. Moreover $|N_G(\bar{T})|=4|\bar{T}|$ and $C_G(j)|=\frac{1}{2}q(q-1)^2(q+1)$. Therefore, by Lemma \ref{reflexive}, the elements of the form $ii^g$, which have even order, is at least
\begin{eqnarray*}
\frac{|S||T||C_G(j)|^2}{2|N_G(T)||G|}& = & \frac{q^2(q-1)^6(q+1)^2}{128q^4(q^2-1)(q^4-1)}\\
&=&\frac{(q-1)^4}{128q^2(q^2+1)}\\
&\geqslant& \frac{1}{128}\cdot \frac{1}{3}=\frac{1}{384}
\end{eqnarray*}
since $q\geqslant 5$. Note that at least half of the elements $\zeta_0^j(g)$ belong to only $T_1$ or $T_2$ and the result follows. \hfill $\Box$
\end{proof}

\section{Preliminary algorithms}
\

\subsection{Probabilistic recognition of classical groups}
\
     
A probabilistic recognition algorithm for finite simple groups of Lie type, that is, computation of their standard names, is presented in \cite{babai02.383} by using the order oracle. The idea is based on the analysis of the statistics of element orders, which are specific for each group of Lie type except for the groups $\psp_{2n}(q)$ and $\psizo_{2n+1}(q)$. This approach fails to distinguish these two classes of groups since, especially when the size of the field is large,  random elements are regular semisimple with probability close to $1$ and the statistics of orders of regular semisimple elements are virtually the same for $\psp_{2n}(q)$ and $\psizo_{2n+1}(q)$, see \cite{altseimer01.1} for thorough discussion. To complete the recognition problem for all finite simple groups of Lie type Altseimer and Borovik presented an algorithm distinguishing $\psp_{2n}(q)$ from $\psizo_{2n+1}(q)$,  $q$ odd, by using the centralisers of involutions and conjugacy classes in these groups \cite{altseimer01.1}. 

We present an alternative probabilistic recognition algorithm for black box classical groups of odd characteristic. The algorithm determines the type of the given black box classical group $G$, that is, it decides whether $G$ is  linear, unitary, symplectic or orthogonal without using the order oracle. This procedure is necessary in the construction of the Curtis-Phan-Tits system of $G$, see Remark \ref{abdnvscn}.

Let $p$ be prime and $k\geqslant 2$, then there is a prime dividing $p^k-1$ but not $p^i-1$ for $1\leqslant i<k$, except when either $p=2$, $k=6$, or $k=2$ and $p$ is a Mersenne prime. Such a prime is called \textit{primitive prime divisor} of $p^k-1$. In our algorithm we are concerned with the primitive prime divisors of $q^a-1$ where $q=p^k$ for some $k\geqslant 1$. It is clear from the definition that each primitive prime divisor of $p^{ak}-1$ is a primitive prime divisor of $q^a-1$. An integer which is a primitive prime divisor of $q^a-1$ is said to have \textit{primitive prime divisor rank} $a$. If the order of a group element $g$ has primitive prime divisor rank $a$, then we say that $g$ has primitive prime divisor rank $a$ and we write $\mbox{pdrank}(g)=a$.

Our algorithm is based on the following result.

\begin{theorem}\label{thm:type}
Let $G$ be finite simple classical group of odd characteristic and $K$ be a long root $\sl_2(q)$-subgroup. Let $L=\la K, K^g \ra$ for a random element $g \in G$. Then, with probability at least $1-2/q$, the followings hold.
\begin{enumerate}
\item If $G\cong \psl^\eps_n(q)$, then $L\cong \ppsl_4^\eps(q)$ for $n \geqslant 4$; $L=G$ for $n=2,3$.
\item If $G\cong \psp_{2n}(q)$, then $L\cong \ppsp_4(q)$ for $n\geqslant 2$.
\item If $G\cong \psizo_{2n+1}(q)$ or $\pso_{2n}^\pm(q)$, then $L\cong \ppso^+_8(q)$ or $L=G\cong \pso_8^-(q)$ for $n\geqslant 4$; $L=G$ for $n\leqslant 3$.
\end{enumerate}
\end{theorem}

\begin{proof}
This is combination of the results presented in \cite[Section 3]{suko02}.
\end{proof}

\begin{theorem}\label{alg:type}
Let $G$ be a simple black box classical group of odd characteristic. Then there exists a polynomial time Monte--Carlo algorithm which computes the type of $G$.
\end{theorem}

\begin{proof}
Let $G\cong \psl_n^\eps(q), \psp_{2n}(q)$, $\psizo_{2n+1}(q)$ or $\pso_{2n}^\eps(q)$. We construct a long root $\sl_2(q)$-subgroup $K$ in $G$ by \cite[Theorem 1.1]{suko02} and take a random element $g\in G$. Then, with probability at least $1-2/q$, the structure of the subgroup $L= \langle K , K^g \rangle$ is determined by Theorem \ref{thm:type}. 

If $n\geqslant 4$, then, by Theorem \ref{thm:type}, $L\cong \ppsl_4^\eps(q), \ppsp_4(q)$ or $\ppso_8^\pm(q)$. Consider a subset $S \subset L$ consisting of random elements from $L$ and let $\mbox{pdrank}(L)=\max\{ \mbox{pdrank}(g) \mid g \in S \}$. By applying the same arguments in the proof of Lemma 2.5 in \cite{kantor01.168}, we can find an element $g \in L$ with maximal primitive prime divisor rank with probability bounded from below by constant, see also \cite[Section 6]{niemeyer98.117} for more details about the distribution of these elements. It is easy to see that $\mbox{pdrank}(L) =2,3,4,6$ or 8. Recall that
$$|\sl_4(q)|=q^6(q^2-1)(q^3-1)(q^4-1),$$
$$|\su_4(q)|=q^6(q^2-1)(q^3+1)(q^4-1),$$
$$|\sp_4(q)|=q^4(q^2-1)(q^4-1),$$
$$|\psizo_8^+(q)|=\frac{1}{2}q^{12}(q^4-1)(q^2-1)(q^4-1)(q^6-1),$$
$$|\psizo_8^-(q)|= \frac{1}{2}q^{12}(q^4+1)(q^2-1)(q^4-1)(q^6-1).$$

If $\mbox{pdrank}(L)=8$, then $L\cong \ppso_8^-(q)$. Note that there are at least $|\ppso_8^-(q)|/16$ elements of primitive divisor rank $8$ by \cite[Lemma 2.5 and Section 4.1.5]{kantor01.168}. 

We assume now that $L\ncong \ppso_8^-(q)$. If $\mbox{pdrank}(L)=6$, then $L\cong \ppso_8^+(q)$, $\ppsu_4(q)$ or $L=G\cong \psu_3(q)$, $\psizo_7(q)$. Similarly, the proportion of elements of primitive divisor rank 6 in these groups is at least 1/16. In $\ppso_8^+(q)$, there are elements of order $(q^4-1)/4$ whereas $\ppsu_4(q)$, $\psu_3(q)$ and $\psizo_7(q)$ do not have such elements. Similarly, in $\psizo_7(q)$, there are elements of order $q^3-1$ where as $\ppsu_4(q)$, $\psu_3(q)$ do not have such elements. Note that we do not need to compute the exact orders of the elements. For example, to distinguish  $\psizo_7(q)$ from $\ppsu_4(q)$ and $\psu_3(q)$, we look for an element $g \in L$ satisfying $g^{q^3-1}=1$ but $g^{q^4-1}\neq 1$ and $g^{q^3+1}\neq 1$.

If $\mbox{pdrank}(L)=4$, then $L\cong \ppsl_4(q), \ppsp_4(q)$. In $\sl_4(q)$, there are at least $|\ppsl_4(q)|/16$ elements of order $(q^4-1)/4(q-1)$ by \cite[Lemma 2.5]{kantor01.168} whereas $\ppsp_4(q)$ does not have such elements.

If $\mbox{pdrank}(L)=2$ or 3, then $L=G\cong \psl_2(q)$ or $\psl_3(q)$, respectively. \hfill $\Box$
\end{proof}

An important corollary of Theorem \ref{alg:type} is an alternative algorithm to Altseimer-Borovik algorithm \cite{altseimer01.1} distinguishing the groups $\psp_{2n}(q)$ and $\psizo_{2n+1}(q)$. 

\begin{corollary}
Let $G$ be a black box group isomorphic to $\psp_{2n}(q)$ or $\psizo_{2n+1}(q)$, $q >3$, $q$ odd, $n\geqslant 3$. Then there is a one sided Monte--Carlo polynomial time algorithm which decides whether $G$ is isomorphic to $\psp_{2n}(q)$ or not.
\end{corollary}

\subsection{Recognising classical involutions in black box groups}
\

In this section we present an algorithm which decides whether a given involution in  a black box group is classical or not.

\begin{lemma}\label{direct:sl2}
Let $L$ be a finite quasisimple classical group over a field of odd size $q\geqslant 5$, $L\ncong \ppsl_2(q)$ and $K\cong \sl_2(q)$. Let $G=KL$ be a commuting product of $K$ and $L$. Given an exponent $E$ for $G$ and the value of $q$, there exists a polynomial time Monte-Carlo algorithm which constructs $K$ and $L$.
\end{lemma}

\begin{proof}
The proof follows from Step 4 of the presentation of Algorithm 8.3 in \cite{suko02}. 
\end{proof}

\begin{lemma}\label{all:sl2}
Let $G$ be a commuting product of subgroups isomorphic to $\ppsl_2(q)$, $q\geqslant 5$ odd. Then there exists a Monte-Carlo algorithm which constructs all components of $G$. 
\end{lemma}

\begin{proof}
This is \cite[Algorithm 6.8]{suko02} together with the remark following it. \hfill $\Box$
\end{proof}

\begin{lemma}\label{declassin}
Let $G$ be a simple black box classical group over a field of odd size $q\geqslant 5$ and $i \in G$ be an involution. Given an exponent for $G$ and the value of $q$, there exists a polynomial time Monte-Carlo algorithm which decides whether $i$ is a classical involution or not.
\end{lemma}

\begin{proof}
By \cite[Theorem 1.2]{suko02}, we can check whether a subgroup $K \leqslant G$ isomorphic to $\ppsl_2(q^k)$ is a long root $\sl_2(q)$-subgroup or not. Recall that the long root $\sl_2(q)$-subgroups are indeed isomorphic to $\sl_2(q)$ and, by definition, the unique involution that belongs to a long root $\sl_2(q)$-subgroup is a classical involution in $G$.

Let $C=C_G(i)''$. If $i \in G$ is a classical involution, then $C$ is a commuting product of subgroups $K$ and $L$ where $K\cong \sl_2(q)$ and $L\cong \sl_{n-2}(q)$, $\sp_{2n-2}(q)$, $\sl_2(q)\circ_2\psizo_{2n-3}(q)$ or $\sl_2(q)\circ_2\psizo_{2n-4}^\pm(q)$ when $G\cong \psl_n(q)$, $\psp_{2n}(q)$, $\psizo_{2n+1}(q)$ or $\pso_{2n}^\pm(q)$, respectively. Now, by Lemma \ref{direct:sl2} we construct $K$ and $L$ and check whether $K$ is a long root $\sl_2(q)$-subgroup in $G$ by \cite[Theorem 1.2]{suko02}. 

If $i$ is not a classical involution, then either $C$ is isomorphic to a commuting product of quasisimple classical groups $K$  and $L$ with $K, L\ncong \sl_2(q)$ or $C$ has only one component. Hence the algorithm presented in Lemma \ref{direct:sl2} never returns a subgroup $H$ which is isomorphic to $\sl_2(q)$. We check whether $H \ncong \ppsl_2(q)$ in the following way. If $H\ncong \ppsl_2(q)$, then there are sufficiently elements $h \in H$ such that $h^{q(q^2-1)}\neq 1$. Note that $C$ may have only one component isomorphic to $\psl_2(q^k)$ for some $k\geqslant 1$, for example, if $G\cong \psl_4(q)$ ($q\equiv -1 \mbox{ mod } 4$) or $\psp_4(q)$, then there exists an involution $i \in G$ such that $C=C_G(i)''\cong \psl_2(q^2)$ or $\psl_2(q)$, respectively. Clearly, in such cases $C$ does not have central involutions.
\hfill $\Box$
\end{proof}

\section{Construction of Curtis-Phan-Tits system}
\

The aim of this section is to prove Theorem \ref{cpt:main}. We present the following algorithm. 

\medskip

\begin{tabular}{l}\hline
Algorithm: \verb CPT_Classical \\ \hline
\noindent{\bf Input:}\\
\noindent{$\bullet$ } A black box group $G$ known to be isomorphic to a quasisimple classical \\ group over a field of odd size $q\geqslant 5$.\\
\noindent{$\bullet$ } An exponent $E$ for $G$. \medskip \\ 
\noindent{$\bullet$ } The characteristic $p$ of the underlying field.\\
\noindent{\bf Output:}\\
\noindent{$\bullet$ } Generators for all root $\sl_2(q)$-subgroups which forms a extended Curtis-\\Phan-Tits system for $G$ corresponding to some maximal torus.\\ \hline
\end{tabular}
\medskip

The proof of Theorem \ref{cpt:main} follows from the following three routines.

\noindent{\bf Step 1. } Construction of a long root $\sl_2(q)$-subgroup in $G$; \cite[Theorem 1.1]{suko02}.

\noindent{\bf Step 2. } Identification of the type of $G$; Theorem \ref{alg:type}.

\noindent{\bf Step 3. } Construction  of all root $\sl_2(q)$-subgroups associated with the nodes of the extended Dynkin diagram of the corresponding algebraic group; Sections \ref{alg:an}, \ref{alg:bndn} and \ref{alg:cn}.

\begin{remark}\label{abdnvscn}{\rm
1. Except for the groups $\psp_{2n}(q)$, the structure of the algorithm is, generically, based on constructing a long root $\sl_2(q)$-subgroup which together with a given long root $\sl_2(q)$-subgroup generate a subgroup isomorphic to $\sl_3(q)$ or $\su_3(q)$. This approach fails for the groups $\psp_{2n}(q)$ since the nodes of the extended Dynkin diagram correspond to short root $\sl_2(q)$-subgroups except for the end nodes, see Figure \ref{ecn}. For this reason, we follow a different but simpler approach for the groups $\psp_{2n}(q)$. 

2. By the above remark, we need to know the type of the given black box classical group in order to start constructing the root $\sl_2(q)$-subgroups corresponding to the nodes of the extended Dynkin diagram. One can use a probabilistic recognition algorithm presented in \cite{babai02.383}, which uses order oracle, at the beginning of our algorithm but this algorithm does not distinguish the groups $\psizo_{2n+1}(q)$ and $\psp_{2n}(q)$ in which case one has to use the algorithm presented in \cite{altseimer01.1}. To make the arguments in our algorithm uniform, we use the algorithm presented in Theorem \ref{alg:type}.

3. Note that we find the size of the underlying field $q$ at the end of the Step 1 by applying the algorithm presented in \cite[Section 5.1.3]{suko01}. Note that $q$ is the size of the centre of a long root $\sl_2(q)$-subgroup in $G$, .
}
\end{remark}

\subsection{Groups of type $A_{n-1}$}\label{alg:an}
\

In this subsection, we present an algorithm which constructs all long root $\sl_2(q)$-subgroups in a black box group $G$ isomorphic to $A_{n-1}^\eps(q)=\psl_n^\eps(q)$, $n \geqslant 3$, $q\geqslant 5$ corresponding to the nodes in the extended Dynkin diagram of $\psl_n^\eps(q)$. We present the algorithm for $\psl_n(q)$ and the algorithm for $\psu_n(q)$ can be read along the same steps by changing the notation $\sl$ to $\su$. The algorithm returns an extended Curtis-Tits system for the groups $\psl_n(q)$ and a Phan system for $\psu_{n}$.

\begin{figure}[htbp]
\begin{center}
\includegraphics[scale=1]{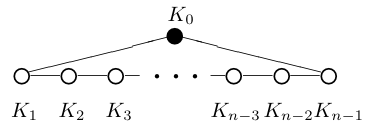}
\caption{Extended Dynkin diagram of $A_n$}
\label{ean}
\end{center}
\end{figure}

\begin{center}
\begin{tabular}{l} \\ \hline
Algorithm: \verb CPT_PSLn \\ \hline
\noindent{\bf 1.} Construct a long root $\sl_2(q)$-subgroup $K_1$ and $C_G(i_1)''=K_1L_1$ where \\ $i_1 \in Z(K_1)$ and $L_1\cong \sl_{n-2}(q)$. \\
\noindent{\bf 2.} Construct a classical involution $i_2\in C_G(i_1)$ where $i_2 \notin C_G(K_1)$. \\ Construct also $C_G(i_2)''=K_2L$ and $K_2$. Set $L_2 = C_{L_1}(i_2)\cong \sl_{n-3}(q).$\\

\noindent{\bf 3.} For $s=3,\ldots , n-1$,  construct  classical involutions $i_s \in C_{L_{s-2}}(i_{s-1})$\\ where $i_s \notin C_G(K_{s-1})$, $C_{L_{s-2}}(i_s)''=K_sL_s$, $K_s$ and $L_s$. Note that  $L_{n-2}=1$.\\

\noindent{\bf 4.} Construct $i_0=i_1i_2 \cdots i_n$, $C_G(i_0)''=K_0L_0$ and $K_0$.
\\ \hline
\end{tabular}
\end{center}

\noindent \textbf{Step 1. Construction of $K_1$} 

We use \cite[Theorem 1.1]{suko02} to construct a long root $\sl_2(q)$-subgroup $K_1\leqslant G$. Let $i_1$ be the unique involution in $K_1$ and $C_1=C_G(i_1)^{\prime \prime} = K_1L_1$ where $L_1 \cong \sl_{n-2}(q)$. By Lemma \ref{direct:sl2},  we can construct $L_1$.

\noindent\textbf{Step 2: Construction of $K_2$} 

By Lemma \ref{conn:sln:dist}, we can find an element $g\in G$ such that $i_2=\zeta_0^{i_1}(g) \notin C_G(K_1)$ is a classical involution with probability at least $1/750$. Recall that, by the definition of the map $\zeta_0^{i_1}$, the involution $i_2 \in C_G(i_1)$. If $G\ncong \psl_4(q)$, then $i_2$ is a classical involution by Lemma \ref{conn:sln}. If $G\cong \psl_4(q)$, then we check whether $i_2$ is a classical involution or not by Lemma \ref{declassin}. Now, assume that $i_2$ is a classical involution and construct $C_2=C_G(i_2)''=K_2L$ where $K_2 \cong \sl_2(q)$ and $L \cong \sl_{n-2}(q)$. By Lemma \ref{direct:sl2}, we can construct $K_2$, and it follows from Lemma \ref{porct:sln} that $\langle K_1,K_2 \rangle \cong \sl_3(q)$.

Since $i_2 \in C_G(i_1)$, we have $i_2 \in N_G(L_1)$ by Theorem \ref{classical:sl2}. Moreover, $i_2$ acts as an involution of type $t_1$ on $L_1$ since it is a classical involution in $G$ and $i_2 \notin L_1$. 
Since $i_2$ acts as an involution of type $t_1$ on $L_1$ and $i_2 \in N_G(L_1)$, we have $L_2 =C_{L_1}(i_2)'' \cong \sl_{n-3}(q)$.

Observe that $i_1$ acts an involution of type $t_1$ on $K_2$. Therefore, if $C_{K_2}(i_1)''=1$, then $G\cong \psl_3(q)$. In this case, we start constructing the subgroup corresponding to the extra node in the extended Dynkin diagram. It is clear that the involution $i_0 = i_1 i_2$ is a classical involution in $G$ satisfying $i_0 \in C_G(i_s)$ and $i_0 \notin C_G(K_s)$ for $s=1,2$. 
Let $K_0=C_G(i_0)''$ be the corresponding long root $\sl_2(q)$-subgroup. Then $\langle K_1, K_2 \rangle = \langle K_2 , K_0 \rangle = \langle K_0, K_1 \rangle \cong \psl_3(q)$ by Lemma \ref{porct:sln}. Hence, the subgroups  $K_0$, $K_1$ and $K_2$ correspond to the nodes in the extended Dynkin diagram.

\noindent\textbf{Step 3: Construction of $K_3, K_4, \ldots, K_{n-1}$} 

Assume that $n\geqslant 4$. We first check whether $L_2 =C_{L_1}(i_2)'' = 1$. By the above construction, if $L_2=1$, then $G\cong \psl_4(q)$. In this case, $L_1\cong L \cong \sl_2(q)$ where $L$ is the subgroup constructed in Step 2. We have $\langle K_1, K_2 \rangle \cong \langle K_2 , L_1 \rangle \cong \langle L_1, L \rangle \cong \sl_3(q)$ by Lemma \ref{porct:sln}, and $[K_1, L_1] = [K_2,L]=1$. Therefore, setting $K_3=L_1$ and $K_4 = L$, the subgroups $K_1, K_2,K_3$ and $K_4$ form a extended Curtis-Tits system for $G$.

Assume now that $n\geqslant 5$ and start working in $L_1\cong \sl_{n-2}(q)$. Since $i_2 \in N_G(L_1)$, we have $i_2i_2^g \in L_1$ for any $g \in L_1$. Moreover, $i_2i_2^g$ has even order with probability at least $1/8$ by Lemma \ref{sln:t1}. Hence we can construct an involution $i_3 = \zeta_0^{i_2}(g) \in L_1$ for some $g \in L_1$ with probability at least $1/8$. By Lemma \ref{sln:t1}, $i_3$ is a classical involution in $L_1$ so, by Lemma \ref{conn:sln}, $i_3 \notin C_G(K_2)$. Now we construct $C_{L_1}(i_3)'' =K_3L_3$ where $K_3 \cong \sl_2(q)$ and $L_3 \cong \sl_{n-4}(q)$. It is clear that $L_3 \leqslant L_2$. By Lemma \ref{direct:sl2}, we can construct $K_3$ and $L_3$, and by Lemma \ref{porct:sln}, $\langle K_2,K_3\rangle \cong \sl_3(q)$. We have
\begin{itemize}
\item $[i_s,i_t]=1$ for $s,t=1,2,3$.
\item $[K_1,K_3]=1$.
\item $\langle K_1, K_2 \rangle \cong \langle K_2,K_3 \rangle \cong \sl_3(q)$.
\end{itemize}

Similarly, we construct a classical involution $i_4=\zeta_0^{i_3}(g) \in L_2$ for some $g\in L_2$ and continue in this way. Notice that $L_{n-3} \cong \sl_2(q)$ and $L_{n-2} =1$. Hence, the last recursion step occurs for a classical involution $i_{n-1}\in C_{L_{n-3}}(i_{n-2})$. Indeed, following the above construction, we have $C_{L_{n-3}}(i_{n-1})'' = K_{n-1}\cong \sl_2(q)$ so $K_{n-1}=L_{n-3}$. Hence, we obtain
\begin{itemize}
\item $[i_s,i_t]=1$ for $s,t=1,\ldots,n-1$.
\item $[K_s,K_t]=1$ for all $s,t=1, \ldots n-1$ with $|s-t|\geqslant 2$.
\item $\langle K_s, K_{s+1} \rangle \cong \sl_3(q)$ for $s=1,\ldots, n-2$.
\end{itemize}

\noindent\textbf{Step 5: Construction of $K_0$} 

Let $i_0=i_1i_2 \cdots i_{n-1}$. It is clear that $i_0$ is a classical involution and it does not centralise $K_1$ and $K_{n-1}$. Moreover, $i_0 \in N_G(K_1)\cap N_G(K_{n-1})$.  Let $K_0$ be the long root $\sl_2(q)$-subgroup containing $i_0$. Then, by construction, $[K_0,K_s]=1$ for $s=2,3,\ldots, n-2$ and $\langle K_0,K_1 \rangle \cong \langle K_0, K_{n-1}\rangle \cong \sl_3(q)$ by Lemma  \ref{porct:sln}. Hence the subgroups $K_0,K_1, \ldots , K_{n-1}$ form an extended Curtis-Tits system for $G$.

\subsection{Groups of type $B_n$ and $D_n$}\label{alg:bndn}
\

In this section we present our algorithm for $B_n(q)=\psizo_{2n+1}(q)$ where $n\geqslant 3$ and $D_n^\eps(q)=\pso_{2n}^\pm(q)$ where $n\geqslant 4$ and $q\geqslant 5$. Recall that $\psizo_5(q) \cong \psp_4(q)$, $\pso_6^+(q) \cong \psl_4(q)$ and $\pso_6^-(q)\cong \psu_4(q)$. We present our algorithm for $B_n(q)$ and $D_n(q)$ separately.

\begin{figure}[htbp]
\begin{center}
\includegraphics[scale=1]{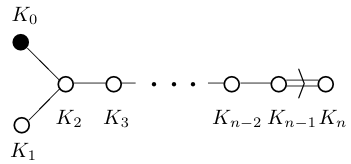}
\caption{Extended Dynkin diagram of $B_n$}
\label{ebn}
\end{center}
\end{figure}

\begin{center}
\begin{tabular}{l}\hline
\noindent Algorithm: \verb CPT_Bn \\ \hline
{\bf 1.} Construct long root $\sl_2(q)$-subgroup $K_1$ and $C_G(i_1)''=K_0K_1L_1$ where \\$i_1\in Z(K_1)$. Construct also $K_0$ and $L_1 \cong\psizo_{2n-3}(q)$.\\
{\bf 2.} Construct a classical involution $i_2 \in C_G(i_1)$ where $i_2 \in N_G(K_1) \backslash C_G(K_1)$. \\ Construct also $C_G(i_2)''=K_2\tilde{K}_2L$ and $K_2$. Set $L_2=C_{L_1}(i_2)'' \cong \psizo_{2n-5}(q)$.\\
{\bf 3.} For $s=3, \ldots , n-1$, construct classical involutions $i_s \in C_{L_{s-2}}(i_{s-1})$ \\ where $i_s \in N_G(K_{s-1})\backslash C_G(K_{s-1})$, $C_{L_{s-2}}(i_s)''=K_s\tilde{K}_sL_s$, $K_s$ and \\ $L_s\cong \psizo_{2(n-s)-1}(q)$.  Note that $L_{n-2} \cong \psizo_3(q) \cong \psl_2(q)$ and  $L_{n-1} =1$.\\
{\bf 4.} 
Set $K_n=L_{n-2}$.\\ 
\hline
\end{tabular}
\end{center}

\noindent{\bf Step 1.} We use \cite[Theorem 1.1]{suko02} to construct a long root $\sl_2(q)$-subgroup $K_1$ in $G$. Let $i_1$ be the unique involution in $K_1$. Then $C_G(i_1)'' =K_0K_1L_1$ where $K_1 \cong K_0\cong \sl_2(q)$ and $L_1 \cong \psizo_{2n-3}(q)$. 

By Lemma \ref{twin}, we can find an element $g \in C_G(i_1)$ such that $\langle K_1, K_1^g \rangle \cong \sl_2(q) \circ_2 \sl_2(q)$ with probability at least $1/8$. It is clear that the components are $K_0$ and $K_1$. By Lemma \ref{all:sl2},  we can construct $K_0$. 

If $n=3$, then $L_1\cong \psizo_3(q) \cong \psl_2(q)$ so $C_G(i_1)''$ is a commuting product of the subgroups $K_0, K_1$  and $L_1$ and we apply Lemma \ref{all:sl2} to construct $L_1$. If $n\geqslant 4$, we apply Lemma \ref{direct:sl2} to construct $L_1$.

\noindent{\bf Step 2.} We construct an involution $i_2 \in G$ by using the map $\zeta_0^{i_1}$ with the property that $i_2 \in N_G(K_1)\backslash C_G(K_1)$. By Lemma \ref{conn:orth}, $i_2$ is an involution of type $t_1, t_2$ (classical) or $t_3$. By Lemma \ref{declassin}, we can decide if $i_2$ is a classical involution. By Lemma \ref{conn:orth:dist}, we can find an element $g \in G$ such that $i_2=\zeta_0^{i_1}(g)$ is a classical involution and $i_2 \in N_G(K_1)\backslash C_G(K_1)$ with probability bounded from below by constant. Now  $C_G(i_2)''= K_2\tilde{K}_2L$ where $K_2 \cong \tilde{K}_2 \cong \sl_2(q)$ and $L \cong \psizo_{2n-3}(q)$. Observe that $i_2$ acts as an involution of type $t_1$ on both $K_1$ and $L_1$. Thus $L_2 = C_{L_1}(i_2)'' \cong \psizo_{2n-5}(q)$. 

By Lemma \ref{porct:orth:I}, $\langle K_1,K_2 \rangle \cong \sl_3(q)$ or $\su_3(q)$. If  $L_2 =1$, then $n=3$ and $L_1 \cong \psizo_3(q)$. In this case, the subgroup $\langle K_2, L_1 \rangle$ acts irreducibly on  a non-degenerate $5$-dimensional orthogonal space. Hence $\langle K_2, L_1 \rangle \cong \psizo_5(q)\cong \psp_4(q)$. Thus, setting $K_3=L_1$, the subgroups $K_0, K_1, K_2, K_3$ form an extended Curtis-Tits or Phan system depending on $\langle K_1, K_2 \rangle \cong  \sl_3(q)$ or $\su_3(q)$, respectively.

\noindent{\bf Step 3.} Assume now that $n\geqslant 4$. Recall that $L_1 \cong \psizo_{2n-3}(q)$. We first construct a classical involution $i_3 \in L_1$ such that $i_3 \in N_G(K_2)\backslash C_G(K_2)$. Since $i_2 \in N_G(L_1)$ and it acts as an involution of type $t_1$ on $L_1$, an element of the form $i_2i_2^g$ has even order for a random element $g \in L_1$ with probability at least $1/960$ by Lemma \ref{orth:t1}. Hence, we can construct a classical involution  $i_3=\zeta_0^{i_2}(g)$ for some $g \in L_1$ with probability at least $1/960$. Since $i_3 \in C_G(i_2)$, by Lemma \ref{conn:orth:dist}, $i_3 \in N_G(K_2)\backslash C_G(K_2)$ with probability bounded from below by constant. Note that $i_3 \in L_1$ since $i_2 \in N_G(L_1)$ and $g\in L_1$. Now $C_{L_1}(i_3)''=K_3\tilde{K_3}L_3$ where $K_3 \cong \tilde{K}_3 \cong \sl_2(q)$ and $L_3\cong \psizo_{2n-7}(q)$. We construct $K_3$ and $L_3$ by using Lemma \ref{direct:sl2}. Since $\langle K_1,K_2 \rangle \cong \sl_3(q)$ or $\su_3(q)$, we have $\langle K_2,K_3 \rangle \cong  \sl_3(q)$ or $\su_3(q)$ by Lemma \ref{porct:orth:II}, respectively. Hence, we start building either the Curtis-Tits system or the Phan system for $G$.

Similarly, we construct classical involutions $i_s \in C_{L_{s-2}}(i_{s-1})$ where  $i_s \in N_G(K_{s-1}) \backslash C_G(K_{s-1})$. We have $C_{L_{s-2}}(i_s)''=K_s\tilde{K}_sL_s$ for $s=4,\ldots , n-1$ where $K_s \cong \tilde{K}_s \cong \sl_2(q)$ and $L_s \cong \psizo_{2(n-s)-1}(q)$.  Notice that $C_{L_{n-1}}(i_{n-1})=K_{n-1}\tilde{K}_{n-1}$, that is, $L_{n-1}=1$ and $L_{n-2} \cong \psizo_3(q)$. Hence, we have
\begin{itemize}
\item $K_0, K_1, \ldots , K_{n-1}$ where $K_s\cong \sl_2(q)$ for $s=0,1,2,\ldots, n-1$. 
\item $\langle K_0,K_2 \rangle$ and $\langle K_s,K_t\rangle$ are all isomorphic to $\sl_3(q)$ or $\su_3(q)$ for $|s-t|=1$, $s,t\geqslant 1$.
\item  $[K_s,K_t]=1$ for $|s-t|\geqslant 2$,  $(s,t)\neq (0,2)$ or $(2,0)$.
\item $\langle K_0,K_1\rangle \cong \sl_2(q)\circ_2\sl_2(q)$. 
\end{itemize}

\noindent{\bf Step 5.} We set $K_n = L_{n-2}$. It is clear that the subgroup $\langle K_{n-1}, K_n \rangle$ acts irreducibly on 5-dimensional non-degenerate orthogonal space, hence $\langle K_{n-1}, K_n \rangle \cong \psizo_5(q)\cong \psp_4(q)$. Thus the subgroups $K_0, K_1, \ldots, K_n$ form an extended Curtis-Phan-Tits system for $G$.

Now we present a Curtis-Phan-Tits system for the groups $\pso_{2n}^+(q)$ where $n\geqslant 4$, $q\geqslant 5$. Recall that $\pso_6^+(q)\cong \psl_4(q)$ and $\pso_4^+(q) \cong \psl_2(q)\times \psl_2(q)$.
\begin{figure}[htbp]
\begin{center}
\includegraphics[scale=1]{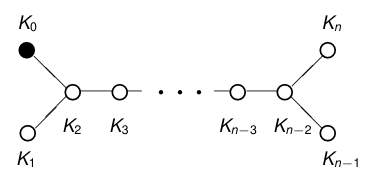}
\caption{Extended Dynkin diagram of $D_n$}
\label{edn}
\end{center}
\end{figure}

\begin{center}
\begin{tabular}{l}\hline
Algorithm: \verb CPT_Dn+ \\ \hline
{\bf 1.} Construct long root $\sl_2(q)$-subgroups $K_0$ and $K_1$, and $C_G(i_1)''=K_0K_1L_1$ \\where $i_1\in Z(K_1)$ and $L_1 \cong\psizo_{2n-4}^+(q)$.\\
{\bf 2.} Construct a classical involution $i_2 \in C_G(i_1)$ where $i_2 \in N_G(K_1)\backslash C_G(K_1)$. \\Construct  also $C_G(i_2)''=K_2\tilde{K}_2L$ and $K_2$. Set $L_2=C_{L_1}(i_2)'' \cong \psizo_{2n-6}^\pm(q)$.\\
{\bf 3.} For $s=3, \ldots , n-2$, construct classical involutions $i_s \in C_{L_{s-2}}(i_{s-1})$ where \\$i_s \in N_G(K_{s-1})\backslash C_G(K_{s-1})$, $C_{L_{s-2}}(i_s)''=K_s\tilde{K}_sL_s$, $K_s$ and $L_s\cong\psizo_{2(n-s)-2}(q)$. \\Note that $L_{n-3} \cong \psizo_4^\pm(q))$ and  $L_{n-2} =1$.\\
{\bf 4.} If $L_{n-3} \cong \psizo_4^+(q)$, then construct the components $K_{n-1}$ and $K_n$. If \\$L_{n-3}=\psizo_4^-(q)$, then we set $K_{n-1}=L_{n-3}\cong \psl_2(q^2)$.\\ \hline
\end{tabular}
\end{center}

\noindent{\bf Step 1,2,3.} Same as in the groups of type $B_n$.

\noindent{\bf Step 4.} We have
\begin{itemize}
\item $K_0, K_1, \ldots , K_{n-2}$ where $K_s\cong \sl_2(q)$ for $s=0,1,2,\ldots, n-2$. 
\item $\langle K_s,K_t\rangle$ and $\langle K_0,K_2 \rangle $  are all isomorphic to $\sl_3(q)$ or $\su_3(q)$ for $|s-t|=1$, $s,t\geqslant 1$.
\item $[K_s,K_t]=1$ for $|s-t|\geqslant 2$,  $(s,t)\neq (0,2)$ or $(2,0)$.
\item $\langle K_0,K_1\rangle \cong \sl_2(q)\circ_2\sl_2(q)$
\end{itemize}

Observe that if $\langle K_s,K_t\rangle$ are all isomorphic to $\sl_3(q)$ for $s,t\geqslant 1$ with $|s-t|=1$, then $L_{n-3}=K_{n-1}\circ_2 K_n \cong \psizo_4^+(q)$ where $K_{n-1} \cong K_n \cong \sl_2(q)$. Moreover $i_{n-1}=i_n$ where $i_{n-1}$ and $i_n$ are the unique involutions in $K_{n-1}$ and $K_n$, respectively. By the construction, it is clear that the involution $i_n$ commute with $i_{n-2}$. Hence, by Lemma \ref{porct:orth:II}, $\langle K_{n-2}, K_{n-1} \rangle \cong \langle K_{n-2}, K_{n} \rangle \cong \sl_3(q)$. Thus we obtain an extended Curtis-Tits system for $G$.

If $\langle K_s,K_t\rangle$ are all isomorphic to $\su_3(q)$ for $s,t\geqslant 1$ with $|s-t|=1$, then $L_1\cong \psizo_{2n-4}^+(q)$, $L_2\cong \psizo_{2n-6}^-(q)$ and so on. In general, $L_s\cong \psizo_{2(n-s)-2}^{\eps_s}(q)$ where $\eps_s=(-1)^{s+1}$. Therefore, if $n$ is even, then $L_{n-3}=K_{n-1}\circ_2 K_n \cong \psizo_4^+(q)$ and we apply the arguments in the previous paragraph to obtain an extended Phan system for $G$. If $n$ is odd, then  $L_{n-3}=\langle K_{n-1}, K_n \rangle \cong \psizo_4^-(q)\cong \psl_2(q^2)$. In this case, we set $K_{n-1}=L_{n-3}$. Note that this is not an extended Phan system for $G$ according to Definition \ref{def:phan}.

The algorithm for the groups $\pso_{2n}^-(q)$ is the same as above except for Step 4. Let $G\cong \pso_{2n}^-(q)$. Applying Step 1, 2 and 3 of the algorithm for the groups $B_n$, we have
\begin{itemize}
\item $K_0, K_1, \ldots , K_{n-2}$ where $K_s\cong \sl_2(q)$ for $s=0,1,2,\ldots, n-2$. 
\item $\langle K_0,K_2 \rangle $ and $\langle K_s,K_t\rangle$ are all isomorphic to $\sl_3(q)$ or $\su_3(q)$ for $|s-t|=1$, $s,t\geqslant 1$.
\item $[K_s,K_t]=1$ for $|s-t|\geqslant 2$,  $(s,t)\neq (0,2)$ or $(2,0)$.
\item $\langle K_0,K_1\rangle \cong \sl_2(q)\circ_2\sl_2(q)$
\end{itemize}

If $\langle K_s,K_t\rangle$ are all isomorphic to $\sl_3(q)$ for $s,t\geqslant 1$ with $|s-t|=1$, then $L_s\cong \psizo_{2(n-s)-2}^-(q)$ for $s=1,2,\ldots, n-2$. 
In particular, $L_{n-3}=\langle K_{n-1}, K_n\rangle  \cong \psizo_4^-(q)\cong \psl_2(q^2)$. In this case, we set $K_{n-1}=L_{n-3}$. Hence we obtain an extended Curtis-Tits system for $G$.

If $\langle K_s,K_t\rangle$ are all isomorphic to $\su_3(q)$ for $s,t\geqslant 1$ with $|s-t|=1$, then $L_1\cong \psizo_{2n-4}^-(q)$, $L_2\cong \psizo_{2n-6}^+(q)$ and so on. In general $L_s\cong \psizo_{2(n-s)-2}^{\eps_s}(q)$ where $\eps_s=(-1)^s$. Therefore, if $n$ is even, then $L_{n-3}=\langle K_{n-1}, K_n\rangle  \cong \psizo_4^-(q)$. In this case, we set $K_{n-1}=L_{n-3}$. Note that this is neither Curtis-Tits nor Phan sytem for $G$, see Definition \ref{def:phan}. If $n$ is odd, then  $L_{n-3}=K_{n-1}\circ_2 K_n \cong \psizo_4^+(q)$ where $K_{n-1} \cong K_n \cong \sl_2(q)$. Moreover $i_{n-1}=i_n$ where $i_{n-1}$ and $i_n$ are the unique involutions in $K_{n-1}$ and $K_n$, respectively. By the construction, it is clear that  $i_{n-1}=i_n$ commute with $i_{n-2}$. Hence, by Lemma \ref{porct:orth:II}, $\langle K_{n-2}, K_{n-1} \rangle \cong \langle K_{n-2}, K_{n} \rangle \cong \su_3(q)$. Thus we obtain an extended Phan system for $G$.

\subsection{Groups of type $C_n$}\label{alg:cn}
\

In this section we present our algorithm for symplectic groups $C_n(q)=\psp_{2n}(q)$ where $n \geqslant 2$, $q \geqslant 5$.

\subsubsection{A small case: $\psp_4(q)$}\label{sub:c2}

The algorithm for $G\cong \psp_4(q)$ is different from the general case at one stage. Therefore, we first construct a Curtis-Phan-Tits system for $G\cong \psp_4(q)$. Note that it is straightforward to distinguish $\psp_4(q)$ from $\psp_{2n}(q)$  for $n\geqslant 3$.

\begin{figure}[htbp]
\begin{center}
\includegraphics[scale=1]{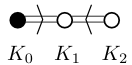}
\caption{Extended Dynkin diagram of $C_2$}
\label{ec2}
\end{center}
\end{figure}

\begin{center}
\begin{tabular}{l} \hline  \
Algorithm: \verb CPT_C2 \\ \hline
\noindent{\bf 1.} Produce an involution $t$ of type $t_2$ and construct $C=C_G(t)$  and \\ $C''\cong \psl_2(q)$.\\
\noindent{\bf 2.}  Construct a classical involution $i \in G$ where $[i,t]=1$.\\
\noindent{\bf 3.}  Construct the components of $C_G(i)'' \cong \sl_2(q) \circ_2 \sl_2(q)$.\\ \hline
\end{tabular}
\end{center}

\noindent\textbf{Step 1:} By Lemma \ref{sptn:tori} and \ref{sptn:dist}, we can construct an involution of type $t_2$ with probability at least $1/10$. In $G$, there are two conjugacy classes of involutions: involutions of type $t_1$ (classical) and of type $t_2$. Therefore, we can decide whether an involution $t\in G$ is of type $t_2$ or not by Lemma \ref{declassin}. Let $t$ be an involution of type $t_2$ and set $K_1=C_G(t)'' \cong \psl_2(q)$. 

\noindent\textbf{Step 2:} We need to construct a classical involution $i \in G$ which commutes with $t$. To do this, we use $\zeta_0^t$, and, by Lemma \ref{conn:sp4}, we can construct such an involution with probability at least $1/768$. Note that we can check whether $i$ is classical or not by Lemma \ref{declassin}. Note also that $[i,t]=1$ by the definition of $\zeta_0^t$.

\noindent\textbf{Step 3:} By Lemma \ref{direct:sl2}, we construct the components of $C_G(i)'' = K_0 \circ_2 K_2$ where $K_0 \cong K_2 \cong \sl_2(q)$. Note that the subgroup $K_1$ stabilises a maximal totally isotropic subspace so $N_G(K_1)$ is a maximal subgroup of $G$ \cite{aschbacher84.469} and $|N_G(K_1)/K_1|=2(q\pm 1)$. Clearly $K_0 \nleqslant N_G(K_1)$ and $K_2 \nleqslant N_G(K_1)$. Hence $\langle K_1,K_0 \rangle =\langle K_1,K_2 \rangle=G$.

\subsubsection{General case: $\psp_{2n}(q)$, $n\geqslant 3$}\

We can now present the algorithm for the general case.

\begin{figure}[htbp]
\begin{center}
\includegraphics[scale=1]{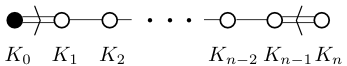}
\caption{Extended Dynkin diagram of $C_n$}
\label{ecn}
\end{center}
\end{figure}

\begin{center}
\begin{tabular}{l} \hline \
Algorithm: \verb CPT_Cn \\ \hline
\noindent {\bf 1.} Construct an element $g \in G$ which has a maximal primitive divisor rank.\\
\noindent {\bf 2.} Construct the involution $i=\mbox{i}(g)$.\\
\noindent {\bf 3.} Check whether $i$ is an involution of type $t_n$. If not, repeat Steps 1 and 2.\\
\noindent {\bf 4.} Construct $C=C_G(i)'' \cong \frac{1}{(2,n)}\sl_n^\eps(q)$.\\
\noindent {\bf 5.} Construct Curtis-Phan-Tits system for $C$, that is, $K_1,K_2, \ldots, K_{n-1}$.\\
\noindent {\bf 6.} Construct $K_n$ and $K_0$.\\ \hline
\end{tabular}
\end{center}

\noindent {\bf Step 1.} Let $T$ be a maximally twisted torus of order $(q^{n}\pm 1)/2$. Then $|N_G(T)/T|=2n$ by \cite[Lemma 2.3]{altseimer01.1}. Therefore, we can find an element $g \in G$ such that  $\mbox{pdrank}(g)=n$ or $2n$ in $O(n)$ random selections from $G$.

\noindent{\bf Step 2.} It is easy to see that involutions which belong to a torus of order $(q^n+1)/2$ are of type $t_n$, see \cite[Lemma 2.13]{altseimer01.1}.

If $q^n\equiv -1 \mbox{ mod 4}$, then we use elements of maximal primitive divisor rank $2n$ to construct involution of type $t_n$. Note that the elements $g \in G$ with $\mbox{pdrank}(g)=2n$ belong to tori of order $(q^n+1)/2$, which is even.

If $q^n\equiv 1 \mbox{ mod 4}$, then $(q^n+1)/2$ is odd. Therefore the elements of maximal primitive divisor rank $2n$ have odd order since they belong to some torus of order $(q^n+1)/2$. In this case, we construct an element $g\in G$ of pdrank $n$, which has even order, and the involution $i=\mbox{i}(g)$. By \cite[Section 4.3]{altseimer01.1}, if $n$ is odd, then $i$ is an involution of type $t_n$. It might happen to be a different type, if $n$ is even.

\noindent{\bf Step 3.} We check whether $i$ is of type $t_n$ in the following way. 
If $i$ is an involution of type $t_n$, then, for $20n$ random elements $g \in G$, one of the elements  $ii^g$ has pdrank $2n$ with probability at least $1-1/e$, see Section 3.4 in the corrected version of \cite{altseimer01.1}. If $i$ is not an involution of type $t_n$, then the $\mbox{pdrank}(ii^g)< 2n$ for any $g \in G$.

Note that if $q^n\equiv -1 \mbox{ mod 4}$ or $n$ is odd, then $i$ is of type $t_n$ by the arguments in Step 2.

\noindent{\bf Step 4.} For an involution $i$ of type $t_n$, if $q \equiv 1 \mbox{ mod } 4$, then $C_G(i)'' \cong \frac{1}{2}\sl_n(q)$, whereas, if $q \equiv -1 \mbox{ mod } 4$, then $C_G(i)'' \cong \frac{1}{2}\su_n(q)$. 

\noindent{\bf Step 5.} In either case $q\equiv \pm 1 \mbox{ mod }4$, we run the algorithm in Subsection \ref{alg:an} to construct the subgroups $K_1, K_2, \ldots , K_{n-1}$. By the description of the centralisers of involutions, see for example \cite[Definition 4.1.8(A) and Table 4.5.1]{gorenstein1998}, $C_G(i)''$ is generated by the all fundamental short root $\sl_2(q)$-subgroups corresponding to a fixed fundamental root system. Hence, we construct all the short root $\sl_2(q)$-subgroups corresponding to the nodes in the Dynkin diagram of $G$, see Figure \ref{ecn}.

\noindent{\bf Step 6.} Let $i_{n-2}$ and $i_{n-1}$ be the unique involutions in $K_{n-2}$ and $K_{n-1}$ respectively. Observe that $C = C_G(i_{n-1})^{\prime\prime}  \cong L_1 L_2$ where $L_1 \cong \sp_4(q)$, $L_2 \cong \sp_{2n-4}(q)$ and $[L_1,L_2]=1$. Moreover, $i_{n-1}\in Z(L_1)$ and $K_{n-1} \leqslant L_1$. Now, since $i_{n-2}$ acts as an involution of type $t_1$ on $K_{n-1}$, it also acts as an involution of type $t_1$ on $L_1$. Hence $C_{L_1}(i_{n-2})''\cong \sl_2(q)\times \sl_2(q)$. It is clear that one of the components isomorphic to $\sl_2(q)$ commutes with $K_{n-2}$ and  we call it $K_n$. By the arguments in Step 3 in Subsection \ref{sub:c2}, we have  $\langle K_{n-1}, K_n \rangle =L_1\cong \sp_4(q)$ and clearly $[K_n,K_j]=1$ for $j=1, \ldots , n-2$. Note that we can construct $L_1$ by a similar procedure as in Step 4 of the presentation of Algrotrihm 8.3 in \cite{suko02}, see also Lemma \ref{direct:sl2}, and $K_n$ by Lemma \ref{all:sl2}. To construct $K_0$, we use $K_1$ and $K_2$ instead of $K_{n-2}$ and $K_{n-1}$, and apply the same method.

\section*{Acknowledgements}
The second author is supported in part by T\"{U}B\.{I}TAK, MATHLOGAPS project  504029 and Australian Research Council Federation Fellowship grant  FF0776186

\end{document}